\documentclass[12pt,psamsfonts]{amsart}
\usepackage{amsmath,amsthm,amssymb,amscd,amsfonts,amsbsy}
\usepackage{subfigure,graphicx,epsfig,overpic}
\usepackage[dvips]{psfrag}
\usepackage{hyperref,url}
\usepackage{array,multirow,multicol}
\usepackage{cite}
\usepackage{color}
\usepackage{epstopdf}
\newtheorem {theorem} {Theorem}
\newtheorem {proposition} [theorem]{Proposition}
\newtheorem {corollary} [theorem]{Corollary}
\newtheorem {lemma}  [theorem]{Lemma}

\newtheorem {remark} [theorem]{Remark}

\newcolumntype{L}{>{\displaystyle}l}
\newcolumntype{C}{>{\displaystyle}c}
\newcolumntype{R}{>{\displaystyle}r}

\newcommand{\R}{\ensuremath{\mathbb{R}}}
\newcommand{\N}{\ensuremath{\mathbb{N}}}
\newcommand{\Z}{\ensuremath{\mathbb{Z}}}
\newcommand{\CO}{\ensuremath{\mathcal{O}}}
\newcommand{\CZ}{\ensuremath{\mathcal{Z}}}
\newcommand{\ov}{\overline}
\newcommand{\la}{\lambda}
\newcommand{\G}{\Gamma}
\newcommand{\T}{\theta}
\newcommand{\f}{\varphi}
\newcommand{\al}{\alpha}

\newcommand{\s}{\ensuremath{\mathbb{S}}}

\newcommand{\C}{\ensuremath{\mathcal{C}}}
\newcommand{\X}{\ensuremath{\mathcal{X}}}
\newcommand{\Y}{\ensuremath{\mathcal{Y}}}
\newcommand{\x}{\mathbf{x}}
\newcommand{\z}{\mathbf{z}}
\newcommand{\dint}{\displaystyle\int}

\def\p{\partial}
\def\e{\varepsilon}
\def\dis{\displaystyle}
\def\l{\omega}

\begin{document}

\title[Limit cycles of piecewise polynomial differential systems]
{Limit cycles of piecewise polynomial\\
perturbations of higher dimensional\\ linear differential systems}

\author[J. Llibre, D.D. Novaes and I.O. Zeli ]
{Jaume Llibre$^1$, Douglas D. Novaes$^{2}$ and Iris O. Zeli$^{3}$}

\address{$^1$ Departament de Matem\`{a}tiques,
Universitat Aut\`{o}noma de Barcelona (Uab), 08193 Bellaterra,
Barcelona, Catalonia, Spain} \email{jllibre@mat.uab.cat}

\address{$^2$ Departamento de Matem\'{a}tica, Universidade
Estadual de Campinas, Rua S\'{e}rgio Buarque de Holanda, 651, Cidade
Universit\'{a}ria Zeferino Vaz, 13083--859, Campinas, SP, Brazil}
 \email{ddnovaes@ime.unicamp.br}
 
 \address{$^2$ Departamento de Matem\'{a}tica, Divis\~ao de Ci\^encias Fundamentais, Instituto Tecnol\'ogico de Aeron\'autica, Pra\c{c}a Marechal Eduardo gomes, 50, Vila das Ac\'acias, 12228--900, S\~ao Jos\'e dos Campos, SP, Brazil}
 \email{iriszeli@ita.br}

\subjclass[2010]{34A36, 34C25, 34C29, 37G15.}

\keywords{limit cycle, averaging method, periodic orbit, polynomial
differential system, nonsmooth polynomial differential systems,
nonsmooth dynamical system, Filippov system.}

\maketitle


\begin{abstract}

The averaging theory has been extensively employed for studying periodic solutions of smooth and nonsmooth differential systems. Here, we extend the averaging theory for
studying periodic solutions a class of regularly perturbed non--autonomous
$n$-dimensional discontinuous piecewise smooth differential
system. As a fundamental hypothesis, it is assumed that the unperturbed system has a manifold $\CZ\subset\R^n$ of periodic solutions satisfying $\dim(\mathcal{Z})<n.$ Then, we apply this result to study limit cycles bifurcating from periodic solutions of
linear differential systems, $x'=Mx$, when they are perturbed inside
a class of discontinuous piecewise polynomial differential systems
with two zones. More precisely, we study the periodic solutions of
the following differential system
\[
x'=Mx+ \e F_1^n(x)+\e^2F_2^n(x),
\]
in $\R^{d+2}$ where $\e$ is a small parameter, $M$ is a
$(d+2)\times(d+2)$ matrix having one pair of pure imaginary
conjugate eigenvalues, $m$ zeros eigenvalues, and $d-m$ non--zero
real eigenvalues.
\end{abstract}

\section{Introduction}

The analysis of discontinuous piecewise smooth differential systems has recently had a large and fast growth due to its applications in several areas of the knowledge. Such systems model many phenomena in control systems (see \cite{Bar}), impact on mechanical systems (see \cite{BSC}), economy (see \cite{Ito}), biology (see \cite{Kri}), nonlinear oscillations (see \cite{Min}), neuroscience (see\cite{Co,HE,NC}), and other fields of science.

Establishing the existence of limit cycles is one of the major problem in the theory of differential systems. The interest in detecting such objects is due to the fact that they are non-local invariant sets providing information on the qualitative behavior of the system. The first studies on this subject considered smooth
differential systems and, since then, many contributions have been made in this direction (see \cite{Ily} and the references
therein). The study of limit cycles has also been considered for continuous (see, for instance, \cite{BL,LMP,LNT})
and discontinuous piecewise smooth differential systems (see, for instance, \cite{GLNP,HY,LM,LMN,LTZ}). Most of them are concentrated on planar piecewise differential systems. 

The {\it averaging theory} is one of the main tools for studying periodic solutions in regularly perturbed differential systems of the form
\begin{equation}\label{introS1}
\dot x=F_0(t,\x)+\sum_{i=1}^k\e^i F_i(t,\x)+\e^{k+1}R(t,\x,\e),\, (t,\x,\e)\in\R\times D\times (-\e_0,\e_0),
\end{equation}
where $D$ is an open bounded subset of $\R^n$ and the functions $F_i,$ $i=0,1,\ldots,k,$ and $R$ are $T$-periodic in the first variable. Here, $k$ is called {\it order of perturbation} in $\e$. As a fundamental hypothesis, it is assumed that the {\it unperturbed system},
\begin{equation}\label{introUS}
\dot x=F_0(t,\x),
\end{equation}
has a manifold $\CZ\subset\R^n$ of periodic solutions. Roughly speaking, this theory provides a sequence of functions, called {\it averaged functions}, which have their simple zeros associated with limit cycles of system \eqref{introS1}. 

The averaging theory has been extensively employed for studying periodic solutions of smooth and nonsmooth differential systems. First, considering $F_0=0$ (consequently, $\dim(\mathcal{Z})=n$) one can find in \cite{V,SVM} results providing sufficient condition on $F_1$ ensuring the existence of periodic solutions of system \eqref{introS1} under smoothness and boundedness conditions. Topological methods were used in \cite{BL}  to generalize these results for Lipschitz continuous differential systems. In \cite{LNT}, assuming the weaker hypothesis $\dim(\mathcal{Z})=n$, the averaging theory was developed at any order for Lipschitz continuous differential systems. Then, in \cite{LNT2015,LMN}, the averaging theory was extended up to order $2$ for detecting periodic orbits of discontinuous piecewise smooth differential systems. Some applications of these results can be found in \cite{LTZ,No}. Finally, in \cite{ILN,LNR}, the averaging theory was developed at any order for a class of discontinuous piecewise smooth systems. 

When $\dim(\mathcal{Z})<n,$ the averaging theory has to be combined with other techniques, for instance \textit{Lyapunov-Schmidt reduction method}, to provide sufficient conditions for the existence of periodic solutions. Here, we also obtain a sequence of function, now called {\it bifurcation functions}, which have their simple zeros associated with limit cycles of system \eqref{introS1}.  In the smooth case, the averaging theory is developed at any order \cite{BuicaLlibreFran,LliNovCand,GineLlibreZhang}. For the nonsmooth case, the first order averaging theory has been addressed in \cite{NovaesThesis15}, however it is lacking in a higher order analysis.

In this paper, our first main goal is to develop the averaging theory up to order $2$ in $\e$ for a class of discontinuous piecewise smooth differential systems assuming $\dim(\mathcal{Z})=d<n.$ The study of any finite order in $\e$ could be performed in a similar way, however the general expression for higher order bifurcation functions would be more complex because it involves higher derivatives of composite functions. As our second main goal, we apply this result to study the number of limit cycles bifurcating from the periodic orbits of a linear differential system $x'=Mx$, where $M$ is a $(d+2)\times(d+2)$ matrix having one pair of pure imaginary conjugate eigenvalues, $m$ zeros eigenvalues, and $d-m$ real eigenvalues. We focus our  attention when this system is perturbed up to order $2$ in the small parameter $\e$ inside a class of discontinuous piecewise polynomial functions having  two zones.

This paper is organized as follows. In Section \ref{sec2}, we state our main results: Theorem \ref{BRt1}, improving the averaging theory for nonsmooth systems; and Theorems \ref{th3}-\ref{th7}, regarding piecewise polynomial perturbations of higher dimensional linear systems. In Section \ref{sec:pr}, we provide some preliminary results. The remainder Sections \ref{secth1}-\ref{phi=2pi} are devoted to the proofs of Theorem \ref{BRt1} and Theorems \ref{th3}-\ref{th7}.

\section{Statements of the main results}\label{sec2}

\subsection{Advances on averaging theory}\label{sub12}

In this subsection we improve the averaging theory of first and
second order to study the limit cycles of a class of discontinuous
piecewise smooth differential systems.

Let $D$ be an open bounded subset of $\R^{d+1}$ and for a positive
real number $T$ we consider the $\C^3$ differentiable functions
$F_i^{\pm}:\s^1\times D\rightarrow\R^{d+1}$ for $i=0,1,2$, and
$R^{\pm}:\s^1\times D\times(-\e_0,\e_0)\rightarrow\R^{d+1}$ where
$\s^1\equiv\R/(\Z T)$. Thus, we define the following $T$-periodic
{\it discontinuous piecewise smooth differential system}
\begin{equation}\label{BRs1}
\x'=\left\{\begin{array}{l}
F^+(\T,\x,\e) \quad \textrm{if}\quad 0\leq\T\leq\phi,\vspace*{-0.2cm}\\
F^-(\T,\x,\e) \quad \textrm{if}\quad \phi\leq\T\leq T,
\end{array}\right.
\end{equation}
where the prime denotes derivative with respect to the variable
$\T\in \s^1$, and
\[
F^{\pm}(\T,\x,\e)=F_0^{\pm}(\T,\x)+\e
F_1^{\pm}(\T,\x)+\e^2F_2^{\pm}(\T,\x)+\e^3R^{\pm}(\T,\x,\e),
\]
with $\x \in D$. The set of discontinuity of system \eqref{BRs1} is
given by $\Sigma=\{\T=0\}\cup\{\T=\phi\}$.

For $\z\in D$, let $\f(\T,\z)$ be the solution of the   unperturbed
system
\begin{equation}
\label{ups} \x'=F_0(\T,\x),
\end{equation}
such that $\f(0,\z)=\z$, where
$$
F_0(\T,\x)=\left\{\begin{array}{l}
F_0^+(\T,\x) \quad \textrm{if}\quad 0\leq\T\leq\phi,\vspace*{-0.2cm}\\
F_0^-(\T,\x) \quad \textrm{if}\quad \phi\leq\T\leq T.
\end{array}\right.
$$
Clearly,
$$
\varphi(\T,\z)=\left\{\begin{array}{l}
\varphi^+(\T,\z) \quad \textrm{if}\quad 0\leq\T\leq\phi,\vspace*{-0.2cm}\\
\varphi^-(\T,\z) \quad \textrm{if}\quad \phi\leq\T\leq T,
\end{array}\right.
$$
where $\f^{\pm}(\T,\z)$ are the solutions of the systems
\begin{equation}\label{pns}
\x'=F_0^{\pm}(\T,\x),
\end{equation}
such that $\f^{\pm}(0,\z)=\z$.

We assume that there exists a manifold $\CZ$ embedded in $D$ such
that the solutions starting in $\CZ$ are all $T$-periodic.
More precisely, for $p=d+1,$ $q\leq p$ and $V$ an open bounded subset of $\R^q$, let $\sigma:\ov
V\rightarrow\R^{p-q}$ be a $\C^3$ function and define
\begin{equation}
\label{Z} \CZ=\{\z_{\nu}=(\nu,\sigma(\nu)):\,\nu \in\ov V\}.
\end{equation}
We shall assume that
\begin{itemize}
\item[($H$)] $\CZ\subset D$ and for each $\z_{\nu}$ the unique solution
$\f(\T,\z_{\nu})$ such that $\f(0,\z_{\nu})=\z_{\nu}$ is $T$-periodic.
\end{itemize}

For $\z\in D$ we consider the first order variational equations of
systems \eqref{pns} along the solution $\f^{\pm}(\T,\z)$, that is
\begin{equation}
\label{lin} Y'=D_{\x}F_0^{\pm}(\T,\f^{\pm}(\T,\z))\,Y.
\end{equation}
Denote by  $Y^{\pm}(\T,\z)$ a fundamental matrix of the differential
system \eqref{lin}.

Let $\xi:\R^q\times \R^{p-q}\rightarrow \R^q$ and
$\xi^{\perp}:\R^q\times \R^{p-q}\rightarrow \R^{p-q}$ be the
orthogonal projections onto the first $q$ coordinates and onto the
last $p-q$ coordinates, respectively. For a point $\rm \z \in D$
denote ${\z}=(u,v) \in  \R^q \times \R^{p-q}$.  Before defining the 
bifurcation functions we have to define some auxiliar functions. Let
\begin{equation}\label{yi}
\begin{array}{l}
y^{\pm}_0(\T,\z)=\f^{\pm}(\T,\z),\vspace{0.2cm}\\
\dis y^{\pm}_1(\T,\z)= Y^{\pm}(\T,\z)\int_0^{\T} Y^{\pm}(s,\z)^{-1}
F^{\pm}_1(s,\f^{\pm}(s,\z))ds,\vspace{0.2cm}\\
y^{\pm}_2(\T,\z)=Y^{\pm}(\T,\z)\dis\int_0^{\T}
Y^{\pm}(s,\z)^{-1}\Bigg(2
F^{\pm}_2(s,\f^{\pm}(s,\z))+\vspace{0.2cm}\\
\qquad \qquad \,\,\, \dis ~2\dfrac{\p F^{\pm}_1}{\p \x}(s,\f^\pm(s,\z))
y^{\pm}_1(s,\z) +\dfrac{\p^2 F^{\pm}_0}{\p \x^2}(s,\f^{\pm}(s,\z))
y^{\pm}_1(s,\z)^{2}\Bigg)ds.
\end{array}
\end{equation}
In the formula of $y^{\pm}_2(\T,\z)$, the second derivative $\dfrac{\p^2 F^{\pm}_0}{\p \x^2}(s,\f^{\pm}(s,\z))$
is a bilinear form defined on $\R^{p}\times \R^{p}$ which is applied to a ``product'' of two vectors, in our case $y^{\pm}_1(s,\z)^{2}$. 

Now, consider
\begin{equation}\label{gi}
g_i(\z)=y_i^+(\phi,\z)-y_i^-(\phi-T,\z),~\text{for $i=0,1,2$}.
\end{equation}
The functions $g_1$ and $g_2$ are usually called averaged functions of order $1$ and $2,$ respectively.
Finally, assuming that the lower right corner
$(p-q)\times(p-q)$ matrix of $Y^+(\phi,\nu)-Y^-(\phi-T,\nu),$ denoted by $\Delta_{\nu},$ is invertible, we define
\begin{equation}\label{gamma}
\gamma(\nu)=-\Delta_{\nu}^{-1}\xi^{\perp}g_1(\z_{\nu}).
\end{equation}

Hence, the bifurcation functions  $f_1, f_2:\ov V\rightarrow \R^q$ of order $1$ and $2$ are given, respectively, by
\begin{align}
\label{fi}
\begin{array}{l}
f_1(\nu)=\xi g_1(\z_{\nu}),\\
f_2(\nu)=    2\dfrac{\p \xi g_1}{\p
v}(\z_{\nu})\gamma(\nu)\vspace{0.2cm} +\dfrac{\p^2 \xi g_0}{\p
v^2}(\z_{\nu})\gamma(\nu)^2 + 2\xi g_2(\z_{\nu}).
\end{array}
\end{align}
Again, in the formula of $f_2$, the second derivative $\dfrac{\p^2 \xi g_0}{\p
v^2}(\z_{\nu})$ is a bilinear form defined on $\R^{(p-q)}\times\R^{(p-q)}$. Thus, as before, we say that it is applied to a ``product'' of two vectors, in our case, $\gamma(\nu)^2.$ 

Our main result on the periodic solutions of system \eqref{BRs1} is
the following.

\begin{theorem}\label{BRt1}
In addition to hypothesis $(H),$ we assume that for any $\nu\in\ov V$
the matrix $Y^+(\phi,\nu)-Y^-(\phi-T,\nu)$ has in the upper right
corner the null $q\times(p-q)$ matrix, and in the lower right corner
has the $(p-q)\times(p-q)$ matrix $\Delta_{\nu}$ with
$\det(\Delta_{\nu})\neq0$. Then, the following statements hold.
\begin{itemize}
\item[$(a)$] If there exists ${\nu^*}\in V$ such that $f_1({\nu^*})=0$ and
$\det(f_1'({\nu^*}))\neq0$, then for $|\e|>0$ sufficiently small
there exists a $T$--periodic solution $\x(\T,\e)$ of system
\eqref{BRs1} such that $\x(0,\e)\to \z_{\nu^*}$ as $\e\to 0$.
\vspace{0.15cm}

\item[$(b)$] Assume that $f_1\equiv 0$. If there exists ${\nu^*}\in V$
such that $f_2({\nu^*})=0$ and $\det(f_2'({\nu^*}))\neq0$, then for
$|\e|>0$ sufficiently small there exists a $T$--periodic solution
$\x(\T,\e)$ of system \eqref{BRs1} such that $\x(0,\e)\to
\z_{\nu^*}$ as $\e\to 0$.
\end{itemize}
\end{theorem}

Theorem \ref{BRt1} is proved in Section \ref{secth1}. The following
result is an immediate consequence of Theorem \ref{BRt1}.

\begin{corollary}\label{cor2}
Assume the hypothesis $(H)$ and that $q=p$, in this case $\CZ=\ov
V\subset D$ is a compact bounded $p$--dimensional manifold. Then, 
statements $(a)$ and $(b)$ of Theorem \ref{BRt1} hold by taking
$f_1=g_1$ and $f_2=2g_2$.
\end{corollary}

\subsection{Perturbations of higher dimensional linear systems}

Consider a $(d+2)\times(d+2)$ matrix  $M$ given by
\begin{equation*}
\label{M}
    M= \left( \begin{matrix}
        0&-1&0_{1\times d} \vspace*{-0.2cm}\\
        1&0&0_{1\times d}\vspace*{-0.2cm}\\
        0_{d\times 1}&0_{d\times 1}& \widetilde M
    \end{matrix} \right),
\end{equation*}
where $0_{i\times j}$ denotes a null $i\times j$ matrix. When  $0<m<d$ assume that  $\widetilde M$ is the diagonal matrix
diag($\mu_1$, $\mu_2$, $\ldots$, $\mu_d$) with $\mu_1=\ldots=
\mu_{m}=0$ and $\mu_{m+1} \neq 0,\ldots, \mu_d \neq 0$. If $m=0$,
then  $\widetilde{M}$ is a diagonal matrix with all entries distinct
from zero, and if $m=d$ we assume that $\widetilde{M}$ is the null
matrix.

Let $L_1=\{(x,0,z):\, x\ge 0, \, {\rm z}\in\R^d\}$ and $L_2=\{(\la
\cos\phi,\la \sin\phi,z):\,\la\ge 0,\, {\rm z}\in\R^d\}$ be two
half--hyperplanes of $\R^{d+2}$ sharing the boundary
$\{0,0,z):\,{\rm z}\in\R^d\}$. The set $\Sigma=L_1\cup L_2$ splits
$D \subset\R^{d+2}$ in $2$ disjoint open sectors, namely $C^+$ and
$C^-$ (see Figure \ref{fig1}).

\begin{figure}[h]
\begin{center}
\begin{overpic}[width=5cm]{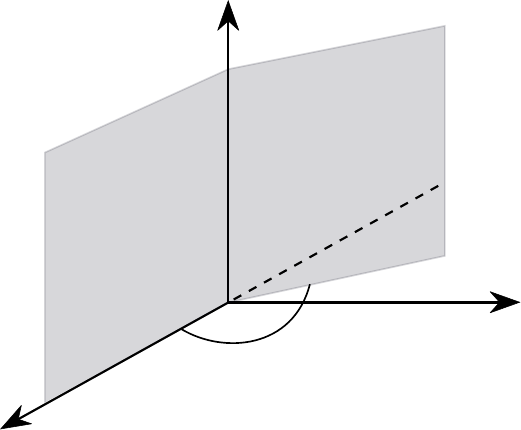}
\begin{small}
\put(55,15){$\phi$} \put(9,0){$x$} \put(96,19){$y$}
\put(47,79){${\rm z}$} \put(20,45){$L_1$} \put(70,65){$L_2$}
\put(30,66){$\Sigma$}
\end{small}
\end{overpic}
\end{center}
\caption{Set of discontinuity $\Sigma$.}
\label{fig1} 
\end{figure}
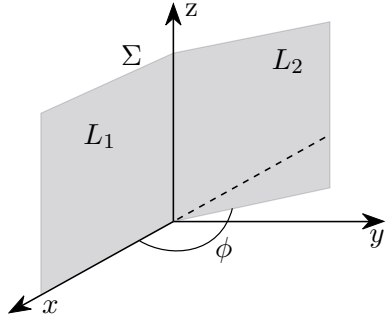

We will denote by $X_\lambda$ and $Y_\lambda$ two polynomials of
degree $n$ in the variables $x,y \in \R$ and  ${\rm z}=(z_1, \ldots,
z_d) \in \R^d$, more precisely
\begin{equation*}
\begin{array}{ll}
X_{\lambda} (x,y,{\rm z}) =& \dis\sum_{i+j+k_1+\ldots+k_d=0}^{n}
\lambda_{ijk_1\ldots k_d} \, x^i y^j z_1^{k_1} \ldots z_d^{k_d},
~\text{and} \vspace*{0.15cm}\\
Y_{\lambda} (x,y,{\rm z}) =&  \dis\sum_{i+j+k_1+\ldots+k_d=0}^{n}
\lambda_{ijk_1\ldots k_d} \, x^i y^j z_1^{k_1} \ldots z_d^{k_d},
\end{array}
\end{equation*}
for  $\lambda_{ijk_1\ldots k_d}\in\R$ and $i,j, k_1, \ldots, k_d \in
\N$. Then, take
\begin{equation}\label{pert}
X^\pm=(X_{a^\pm},X_{b^\pm}, X_{c_{1}^\pm}, \,\ldots\,,
X_{c_{d}^\pm}), \quad Y^\pm=(Y_{\al^\pm},Y_{\beta^\pm},
Y_{\gamma_{1}^\pm}, \,\ldots\,, Y_{\gamma_{d}^\pm}),
\end{equation}
and let $\X(x,y,\mathrm{z})$ and $\Y(x,y,\mathrm{z})$ be polynomial
vector fields defined by
\begin{align*}
\X(x,y,\mathrm{z}) &= X^\pm(x,y,\mathrm{z}) \quad \textrm{if} \quad
(x,y,\mathrm{z})\in C^\pm,\\
\Y(x,y,\mathrm{z})&=  Y^\pm(x,y,\mathrm{z}) \quad\,  \textrm{if}
\quad (x,y,\mathrm{z})\in C^\pm.
\end{align*}
Now, consider the discontinuous piecewise polynomial differential
systems
\begin{equation}
\label{As1} (\dot x, \dot y, \dot{\mathrm{z}})= M
\left(x,y,\mathrm{z}\right) + \e \X  \left(x,y,\mathrm{z}\right) +
\e^2 \Y\left(x,y,\mathrm{z}\right),
\end{equation}
where $x,y\in \R$  and $\mathrm{z}=(z_1,z_2,\ldots,z_d) \in \R^d$.
The dot denotes derivative with respect to the time $t$, and
$\Sigma$ denotes the set of discontinuity for system \eqref{As1}. Also, $M \left(x,y,\mathrm{z}\right)$ is an abuse of notation and denotes the matrix $M$ applied to the vector $\left(x,y,\mathrm{z}\right)$, which is defined as the product between the matrix $M$ with the column matrix associated with the vector $\left(x,y,\mathrm{z}\right)$.  This abuse of notation will be recurrent throughout the paper.

Denote by $N_i(m,n,\phi)$ the maximum number of limit cycles of
system \eqref{As1} that can be detected using averaging theory of
order  $i$ when $|\e|\neq0$ is sufficiently small.

\begin{theorem}\label{th3}
Assume $0\leq m\leq d$, $n\in\mathbb{N}$, and $\phi\in(0,2\pi)\setminus\{\pi\}$. Then, 
\begin{itemize}
\item[$(a)$] $N_1(m,n,\phi)=n^{m+1}$ and \vspace{0.2cm}
\item[$(b)$] $2n(2n-1)^m\leq N_2(m,n,\phi)\leq (2n)^{m+1}$.
\end{itemize}
\end{theorem}
Theorem \ref{th3} generalizes the particular case $m=d$ of
\cite{LTZ}. Comparing itens $(a)$ and $(b)$ of Theorem \ref{th3}, we can easily check that $N_2(m,n,\phi)>N_1(m,n,\phi)$ for every $0\leq m\leq d$, $n\in\mathbb{N}$, and $\phi\in(0,2\pi)\setminus\{\pi\}$. 

Notice that, the lower and upper bounds given in statement $(b)$ of Theorem
\ref{th3} coincide for $m=0$. In this case, $N_2(0,n,\phi) = 2n$. In general, the lower bound of statement $(b)$ of Theorem \ref{th3} is not optimal and can be improved in some cases (see Proposition \ref{improve}).


Theorems \ref{th3} is proved in section
\ref{phineqpi}.

If $\phi = \pi$ we note that the maximum number of limit cycles eventually
decreases as stated in the following result.

\begin{theorem} \label{th6}
Assume $0\leq m\leq d$ and $\phi =\pi$. Then, 
\begin{itemize}
\item[$(a)$] $N_1(m,n,\pi)= n^{m+1}$ and\vspace{0.2cm}
\item[$(b)$] $ N\leq N_2(m,n,\pi) \leq  (2n)^{m+1}$ where
$N=(2n-1)^{m+1}$ if $n$ is odd, and \vspace{0.1cm} \\
$N=(2n-2)(2n-1)^{m}$ if $n$ is even.
\end{itemize}
\end{theorem}

Theorem \ref{th6} is proved in Section \ref{phi=pi}.

Comparing itens $(a)$ and $(b)$ of Theorem \ref{th6}, we can easily check that $N_2(m,n,\pi)\geq N_1(m,n,\pi)$ for every $0\leq m\leq d$ and $n\in\mathbb{N}$, with strictly inequality for $n\neq 1$.

When $\phi = 2\pi$, system \eqref{As1} is continuous. In this case
$\X(x,y,\mathrm{z}) = X^+(x,y,\mathrm{z})$ and $\Y(x,y,\mathrm{z})=
Y^+(x,y,\mathrm{z})$. So, we get the following result.

\begin{theorem}\label{th7}
Assume that $0\leq m\leq d$ and $\phi=2\pi$. Then, 
\begin{itemize}
\item[$(a)$] $N_1(m,n,2\pi)=n^{m}(n-1)/2$ for all $m\neq 0$, and
$$
N_1(0,n,2\pi)= \left\{
\begin{array}{l}
\dfrac{n-1}{2} \quad
\text{if $n$ is odd},\vspace{0.1cm}\\
\dfrac{n-2}{2} \quad \text{if $n$ is even}.
\end{array}\right.
$$
\item[$(b)$]  $n \leq N_2(0,n,2\pi) \leq 2n$.
\end{itemize}
\end{theorem}

Theorem \ref{th7} generalizes the particular cases $m=d=0$ and
$m=d=1$  of \cite{CLT} (see Theorems 2 and 3). Moreover, statement
$(a)$ of Theorem \ref{th7} also generalizes Theorem 1 of \cite{LTZ}
when $m=d$. We prove Theorem \ref{th7} in Section \ref{phi=2pi}.

\section{Preliminary results}\label{sec:pr}

In this section we present some preliminaries results that we shall
need in Sections \ref{phineqpi}, \ref{phi=pi} and \ref{phi=2pi}. In
Section \ref{sub3.1}, we present a change of coordinates so that
system \eqref{As1} reads in the standard form \eqref{BRs1} to apply the averaging method. In Section
\ref{sub3.2}, we construct the averaging functions $f_1$ and $f_2$
for system \eqref{As1}, defined in \eqref{fi}. Finally, in
Section \ref{sub.trig}  we present some trigonometric relations
that will be used in the calculus of the zeros of the functions
$f_1$ and $f_2$.

\subsection{Standard form}\label{sub3.1}

Let $x,y\in\R$ and ${\rm z}=(z_1,\ldots,z_d)\in\R^d$. Using the
change of variables
\begin{equation}
\label{cyl} x=r\cos\T \quad \text{and} \quad y=r\sin\T ,
\end{equation}
with $r\in \R_+$ and  $\T \in \s^1\equiv  \R/(2\pi \Z)$, system
\eqref{As1} becomes
\begin{equation}
\label{As2} \left(\dot \T,\dot r,\dot{{\rm z}}
\right)=(1,0,\widetilde M{\rm z} ) +\e A(\T,r,{\rm z})+  \e^2
B(\T,r,{\rm z}),
\end{equation}
where $A, B:  \s^1 \times \R_+ \times \R^d \to \R^{d+2}$ are
piecewise smooth functions given by
$$
A=\left\{\begin{array}{l}
A^+ \quad \textrm{if}\quad 0\leq\T\leq\phi, \vspace*{-0.2cm}\\
A^- \quad \textrm{if} \quad\phi\leq\T\leq 2\pi,
\end{array}\right.  \quad \textrm{and} \quad
B=\left\{\begin{array}{l}
B^+ \quad \textrm{if}\quad  0\leq\T\leq\phi, \vspace*{-0.2cm}\\
B^- \quad \textrm{if} \quad\phi\leq\T\leq 2\pi,
\end{array}\right.
$$
where
\begin{equation*}
\begin{array}{ll}
A^\pm(\T,r,{\rm z})=& (A^\pm_1(\T,r,{\rm z}), \ldots, A^\pm_{d+2}(\T,r,{\rm z})), \\
B^\pm(\T,r,{\rm z})=& (B^\pm_1(\T,r,{\rm z}),  \ldots,
B^\pm_{d+2}(\T,r,{\rm z})),
\end{array}
\end{equation*}
with
\begin{equation}
\label{AB}
\begin{array}{ll}
A^\pm_1 = & \dfrac{1}{r} \big( X_{b^\pm} \left(r\cos\T,r\sin\T,{\rm z} \right)
\cos\T  - X_{a^\pm} \left(r\cos\T, r\sin\T, {\rm z}\right)\sin\T \big), \vspace*{0.15cm}\\
B^\pm_1 = & \dfrac{1}{r} \big( Y_{\beta^\pm}
\left(r\cos\T,r\sin\T,{\rm z} \right)
\cos\T  - Y_{\al^\pm} \left(r\cos\T, r\sin\T, {\rm z} \right)\sin\T \big), \\
A^\pm_2  =& X_{a^\pm}\left(r\cos\T,r\sin\T,z\right) \cos\T +  X_{b^\pm} \left(
r\cos\T, r\sin\T,{\rm z} \right)\sin\T , \\
B^\pm_2 = & Y_{\alpha^\pm}\left(r\cos\T,r\sin\T,{\rm z} \right)
\cos\T  +
Y_{\beta^\pm} \left(r\cos\T, r\sin\T, {\rm z}\right)\sin\T , \\
A^\pm_{\ell+2}   = & X_{c_{\ell}^\pm}(r\cos\T, r\sin\T,{\rm z}), \\
B^\pm_{\ell+2}  =& Y_{\gamma_\ell^\pm}(r\cos\T, r\sin\T,{\rm z}),
\end{array}
\end{equation}
for $1 \leq \ell \leq d$.  Clearly the discontinuity $\Sigma$ is now
given by
$$
\Sigma= \{(0,r,{\rm z}): r\in \R_+,{\rm z}\in\R^d\} \cup \{
(\phi,r,{\rm z}):r\in\R_+, {\rm z}\in\R^d\}.
$$

Taking the angle $\T$ as the new time, system \eqref{As2} reads
\begin{equation}
\label{eq.sys1}
\begin{array}{l}
r'=\dfrac{\dot r}{\dot \T}=\dfrac{\e A_2(\T,r,{\rm z}) + \e^2
B_2(\T,r,{\rm z})}
{1+\e A_1(\T,r,{\rm z}) + \e^2 B_1(\T,r,{\rm z})},\vspace{.3cm} \\
z'_\ell= \dfrac{\dot{z}_\ell}{\dot{\T}} = \dfrac{ \mu_\ell z_\ell +
\e A_{\ell+2}(\T,r,{\rm z}) +\e^2 B_{\ell+2}(\T,r,{\rm z})} {1+\e
A_1(\T,r,{\rm z}) + \e^2 B_1(\T,r,{\rm z})},
\end{array}
\end{equation}
for $1 \leq \ell \leq d$.  Note that now the prime denotes
derivative with respect to the independent variable $\T$.

Expanding system \eqref{eq.sys1} in Taylor series around $\e=0$,  it
can be written as system \eqref{BRs1} by taking ${\bf x}=(r,{\rm z})
\in D \subset \R_+ \times \R^{d}$ and
\begin{equation}\label{F}
F_{j}^\pm(\T,r,{\rm z}) = (F_{j0}^\pm(\T,r,{\rm z}), \ldots,
F_{jd}^\pm(\T,r,{\rm z})  ), \quad \text{for} \quad j=0,1,2,
\end{equation}
where
\begin{equation}
\label{FF}
\begin{array}{ll}
F_{0\ell}^\pm(\T,r,{\rm z}) =& 0, \\
F_{0\l}^\pm(\T,r,{\rm z}) = &\mu_\l z_\l, \\
F_{1\ell}^\pm(\T,r,{\rm z}) =& A^\pm_{\ell+2}(\T,r,{\rm z}) ,\\
F_{1\l}^\pm(\T,r,{\rm z}) =& A^\pm_{\l +2}(\T,r,{\rm z}) - \mu_
\l z_{\l} A^\pm_1(\T,r,{\rm z})  ,\\
F_{2\ell}^\pm(\T,r,{\rm z}) =&  B^\pm_{\ell+2}(\T,r,{\rm z})
- A^\pm_1(\T,r,{\rm z}) A^\pm_{\ell+2}(\T,r,{\rm z}) ,\\
F_{2\l}^\pm(\T,r,{\rm z}) =& B^\pm_{\l +2}(\T,r,{\rm z})+  \mu_\l z_
{\l } \big(A^\pm_1(\T,r,{\rm z})\big)^2 \\
&- A^\pm_1(\T,r,{\rm z}) A^\pm_{\l +2}(\T,r,{\rm z})- \mu_\l
z_{\l}B^\pm_1(\T,r,{\rm z})  ,
\end{array}
\end{equation}
for $0 \leq \ell \leq m$ and $m+1 \leq \l \leq d$.

When $m=d$ the functions $F^\pm_{j\l}$, for $j=0,1,2$, are not
considered.

\subsection{Construction of the averaging functions}\label{sub3.2}

Now, we shall use the notations introduced in subSection \ref{sub12}.
Since the unperturbed system \eqref{ups} is continuous,  we have
$\varphi^+(\T,{\z})=\varphi^-(\T,{\z})$. Therefore, when $0 \leq m
<d$ the solution of system \eqref{ups} is given by
$$
\varphi(\T,{\z})=(r, z_1, \ldots, z_{m}, e^{ \mu_{m+1}\T} z_{m+1}, \ldots,
e^{\mu_d \T} z_d),
$$
for $\z = (r, {\rm z}) =(r,z_1, \ldots, z_d)$. Note that if
$\z_\nu=(r,z_1,\ldots,z_m,0,\ldots,0)$  then $\varphi(\T,{\z_\nu}) =
\z_\nu$ for every $\T\in\s^1$.  Then, taking an open bounded subset
$V\subset \R^{m+1}$ and the zero function $\sigma: \overline V \to
\R^{d-m}$, the manifold $\CZ$, defined in \eqref{Z}, becomes
$$
\CZ= \lbrace \z_{\nu} =(\nu,0) \in \R^{d+1}: \nu=(r,z_1,\ldots,z_m)
\in  \overline V\rbrace.
$$

For $\z \in D$ a fundamental matrix of system \eqref{lin} is
$$
Y(\T,\z)= \begin{pmatrix}
{\rm Id}_{1+m} &  0 \vspace*{-0.2cm}\\
0&  \Delta
\end{pmatrix},
$$
where ${\rm Id}_{1+m}$ is the $(1+m)\times(1+m)$ identity matrix,
and $\Delta$ is the diagonal matrix diag$(e^{\mu_{m+1} \theta},
\ldots,e^{\mu_{d} \theta})$. Since $Y(\T,\z)$ does not depend of
$\z$  we denote  $Y(\T,\z)=Y(\theta)$. Then, we have
$$
Y(\phi)-Y(\phi - 2\pi) = \left(
\begin{array}{ll}
0 & 0 \vspace*{-0.2cm} \\
0 & \Delta_\nu
\end{array} \right),
$$
where 
\begin{equation} \label{deltanu}
\Delta_\nu = \text{diag}\big(e^{\mu_{m+1}\phi}(1-
e^{-\mu_{m+1}2\pi}), \ldots, e^{\mu_{d}\phi}(1-
e^{-\mu_{d}2\pi})\big).
\end{equation}

According to the notation introduced in Theorem \ref{BRt1} we have
$p=d+1$ and $p-q=d-m$, with $q=m+1$. Since $\CZ$ has dimension
$m+1$, we consider the projections $\xi: \R^{m+1} \times \R^{d-m}
\to \R^{m+1}$ and $\xi^{\bot}: \R^{m+1} \times \R^{d-m} \to
\R^{d-m}$, with $u=(r,z_1, \ldots, z_m) \in \R^{m+1}$ and
$v=(z_{m+1}, \ldots, z_d) \in \R^{d-m}$.

{From} \eqref{yi}  and \eqref{FF} we  have $y_1(\theta,\z)= \big(
y_{10}(\theta,\z), \ldots, y_{1d}(\theta,\z) \big)$ where
\begin{equation}
\label{yii}
\begin{array}{ll}
y_{1\ell}^\pm(\theta,\z)=& \dis \int_0^\T  A_{\ell+2}^\pm (s,\varphi(s,\z))
ds,\vspace*{0.15cm}\\
y_{1\l}^\pm(\theta,\z)=& \dis \int_0^\T  e^{\mu_\l(\T-s)} \big(
A_{\l+2}^\pm (s,\varphi(s,\z))  -  \mu_\l z_{\l
}A_1^\pm(s,\varphi(s,\z)) \big)ds,
\end{array}
\end{equation}
for $0 \leq \ell \leq m$ and $m+1 \leq \l \leq d$. 

Moreover, from
\eqref{gi} we have $g_1(\z_\nu)= \big( g_{10}(\z_\nu), \ldots, g_{1d}(\z_\nu)
\big)$ with
\begin{equation}
\label{gl}
\begin{array}{ll}
g_{1\ell}(\z_\nu)= & \dis \int_0^\phi A_{\ell+2}^+(s,\varphi(s,\z_\nu))\,ds
+ \int_\phi^{2\pi} A_{\ell+2}^-(s,\varphi(s,\z_\nu))\,ds,\vspace*{0.15cm}\\
g_{1\l}(\z_\nu)= &  \dis  \int_0^\phi e^{\mu_\l(\phi -s)}
A_{\l+2}^+(s,\varphi(s,\z_\nu)) \, ds 
+  \int_\phi^{2\pi} e^{\mu_\l(\phi - 2\pi -s)} 
A_{\l+2}^-(s,\varphi(s,\z_\nu)) \, ds
\end{array}
\end{equation}
for $0 \leq \ell \leq m$ and $m+1 \leq \l \leq d$. 

Therefore, the
bifurcation function $f_1: \overline{V}\to \R^{m+1}$, defined in
\eqref{fi},  is given by
\begin{equation}
\label{f1} f_1(\nu)=  \xi g_1(\z_\nu)= \big( f_{10}(\nu) ,\ldots,
f_{1m}(\nu) \big),
\end{equation}
with $f_{1\ell}(\nu)=g_{1\ell}(\z_\nu)$, where $g_{1\ell}$ is given
in \eqref{gl} for $0 \leq \ell \leq m$.

Now, we compute the bifurcation  function $f_2$ defined also in
\eqref{fi}.

Since $g_0$ is linear (see \eqref{yi} and \eqref{gi}) we have
$\dis\frac{\partial^2 \xi g_0}{\partial v^2} (\z_\nu) = 0$.

Moreover, as $\xi^\perp g_1(\z_\nu) = \big( g_{1\, m+1}(\z_\nu), \ldots,
g_{1\, d} (\z_\nu)\big)$, it follows from \eqref{gamma}, \eqref{deltanu} and \eqref{gl}
that  
$$ 
\gamma(\nu)=  \big(\gamma_{m+1}(\nu), \ldots,
\gamma_{d}(\nu) \big),$$
where
\begin{equation}
\label{gammaw}
\begin{array}{l}
\gamma_{\l}(\nu) =   \dfrac{-1}{1 - e^{-\mu_\l 2\pi}} \bigg(
\dint_0^\phi e^{- \mu_\omega s} A_{\l +2}^+  (s,\z_\nu)ds  + \dint_\phi^{2\pi}
e^{-\mu_\omega(2\pi+s)} A_{\l +2}^-(s,\z_\nu)ds\bigg),
\end{array}
\end{equation}
for $m+1\leq \l \leq d$. Furthermore, for $v=(z_{m+1}, \ldots, z_d)$
we have
$$ \dfrac{\partial \xi g_1}{\partial v}
(\z_\nu) \gamma(\nu ) = \big( \widetilde{G}_{10}(\nu) , \ldots,
\widetilde{G}_{1m} (\nu)\big),$$ with
\begin{equation}
\label{G1l} \widetilde{G}_{1\ell}  (\nu)= \dis \sum_{\l=m+1}^{d}
\frac{\partial  g_{1\ell}}{\partial z_{\l}} (\z_\nu)
\gamma_{\l}(\nu),
\end{equation}
where  $g_{1\ell}$ is given in \eqref{gl} for $0 \leq \ell \leq m$.
Additionally from \eqref{gi} and \eqref{yi} we obtain
$$
\xi g_2(\z_\nu)= \xi  (y^+_2(\phi,\z_\nu)) -\xi  (y^-_2 (\phi-
2\pi,\z_\nu)),
$$
where
$$
\xi y^\pm_2(\theta,\z_\nu) =  2 \int_0^\T  \xi
\big(F_2^\pm(s,\z_\nu) \big) +  \xi \bigg(\dis \frac{\p F_1^\pm}{\p
{\bf x}} (s,\z_\nu) y_1^\pm (s,\z_\nu) \bigg) ds,
$$
because  $F_0^\pm$ is linear.

On the other hand
$$
\xi F_2^\pm(s,\z_\nu)= \big({F}_{20}^\pm (s,\z_\nu), \ldots,
{F}_{2m}^\pm (s,\z_\nu) \big),~\text{and}
$$
$$
\xi \bigg( \dis\frac{\p F_1^\pm}{\p {\bf x}} (s,\z_\nu)
y_1^\pm(s,\z_\nu) \bigg) = \big( \widetilde{F}^\pm_{10}(s,\z_\nu) ,
\ldots,  \widetilde{F}^\pm_{1m} (s,\z_\nu) \big),
$$
being
\begin{equation}
\label{Ftil} \widetilde{F}^\pm_{1\ell}(s,\z_\nu) = \dis \frac{\p
F^\pm_{1\ell}}{\p r}(s,\z_\nu)y^\pm_{10}(s,\z_\nu) +     \ldots +
\frac{\p F^\pm_{1\ell}}{\p z_d}(s,\z_\nu)y^\pm_{1d}(s,\z_\nu),
\end{equation}
for $F^\pm_{1\ell}$ and $F^\pm_{2\ell}$ defined in \eqref{FF} for $0
\leq \ell \leq m$. Hence
\begin{equation}\label{f2}
f_2(\nu)= 2\, \dfrac{\p \xi g_1}{\p v}(\z_{\nu})\gamma(\nu) + 2\,
\xi g_2(\z_{\nu})  = \big( f_{20}(\nu), \ldots, f_{2m}(\nu)  \big),
\end{equation}
where
\begin{equation}
\label{f2l1}
\begin{array}{ll}
f_{2\ell}(\nu) =&    2\, \dis\widetilde{G}_{1\ell}(\nu) +\dis  4\int_{0}^{\phi}
\big(F^+_{2\ell}(s,\z_\nu)  + \widetilde{F}^+_{1\ell}(s,\z_\nu) \big)ds
\vspace*{0.15cm}\\
& + ~4\dis\int_{\phi}^{2\pi} \big(F_{2\ell}^-(s,\z_\nu) +
\widetilde{F}^-_{1\ell}(s,\z_\nu)\big) ds,
\end{array}
\end{equation}
for $0\leq \ell \leq m$. See the explicit expression of all
functions that appear in \eqref{f2l1} in the Appendix.

If  $m=d$, then  the functions $\widetilde{G}_{1\ell}(\nu)$ are not
considered because $f_2 =2 g_2$ (see Corollary \ref{cor2}).

\subsection{Some trigonometric integrals}\label{sub.trig}

In order to study the zeros of the averaging functions $f_1$ and
$f_2$, we need to know some results about trigonometric integrals.
Then, we shall state Lemma \ref{lemtrig}. The proof of this lemma
will be omitted here, but it can easily be proven using some
trigonometric relations found in Chapter 2 of \cite{gr}.

For $p,q \in \N$ and $\phi \in (0,2\pi]$ consider the  functions
\begin{equation} \label{IJ}
I_{(p,q,\phi)} = \dint_0^\phi \cos^{p} s\sin^q s  \,ds,\quad
J_{(p,q,\phi)} = \dint_\phi^{2\pi} \cos^{p} s \sin^q s \,ds.
\end{equation}

\begin{lemma}\label{lemtrig}
Let $I_{(p,q,\phi)}$ and $J_{(p,q,\phi)}$ be the functions defined
in \eqref{IJ} for $\phi \in (0,2\pi]$. Then, the following statements
hold.
\begin{itemize}
\item [$(a)$] If $\phi \neq \pi$ and  $\phi \neq 2 \pi$ then $I_{(p,q,
\phi)}$,  $J_{(p,q,\phi)}$, $\dis\int_0^\phi \cos^i s \sin^j s \,
I_{(p,q,\phi)} \, ds$, and  $\dis\int_\phi^{2\pi} \cos^i s \sin^j s
\, I_{(p,q,\phi)} \, ds$ are non--zero;\vspace*{0.3cm}

\item [$(b)$] If $\phi =\pi$ then $I_{(p,q,\pi)}=0$ or $J_{(p,q,\pi)} =
0$  if and only if $p$ is odd.

Moreover
$$
\dis\int_0^\pi \cos^i s \sin^j s \,I_{(p,q, s)} \,ds =0~~~ \text{or}~~
\dis\int_\pi^{2\pi} \cos^i s \sin^j s\,  I_{(p,q,s)} \,ds=0
$$
if and only if one of the following statements hold:
\begin{itemize}
\item [$(i)$] $i$, $j$, $p$ and $q$ are odd;
\item [$(ii)$] $i$, $p$ and $q$ are odd, and $j$ is even;
\item [$(iii)$] $i$ and $p$ are odd, and $q$ and $j$ are  even;
\item [$(iv)$] $i$, $p$ and $j$ are odd, and $q$ is even.
\end{itemize}
\vspace*{0.3cm}
\item [$(c)$] If $\phi =2\pi$ then $I_{(p,q,2\pi) } \neq  0$ if and
only if $p$ and $q$ are simultaneously even.
\end{itemize}
\end{lemma}

\section{Proof of Theorem \ref{BRt1}}\label{secth1}

The proof of Theorem \ref{BRt1} is based on the next lemma which is
a particular case of the {\it Lyapunov-Schmidt reduction} for a
finite dimensional function (see for instance \cite{C}).

\begin{lemma}\label{LS}
Assuming $q\leq p$ are positive integers, let $D$ and $V$ be open
bounded subsets of $\R^p$ and $\R^q$, respectively. Let $g:D\times
(-\e_0,\e_0)\rightarrow\R^p$ and $\sigma:\ov V\rightarrow \R^{p-q}$
be $\C^3$ functions such that $g(\z,\e)=g_0(\z)+\e g_1(\z)+ \e^2
g_2(\z)+\CO(\e^3)$ and $\CZ=\{\z_{\nu}=(\nu,\sigma(\nu)):\,
\nu\in\ov V\}\subset D$. We denote by $\G_{\nu}$ the upper right
corner $q\times (p-q)$ matrix of $D\,g_0(\z_{\nu})$, and by
$\Delta_{\nu}$ the lower right corner $(p-q)\times (p-q)$ matrix of
$D\,g_0(\z_{\nu})$. Assume that for each $\z_{\nu}\in \CZ$,
$\det(\Delta_{\nu})\neq0$ and $g_0(\z_{\nu})=0$. We consider the
functions $f_1, f_2:\ov V\rightarrow \R^q$ defined in \eqref{fi}.
Then, the following statements hold.
\begin{itemize}
\item[$(a)$] If there exists ${\nu^*}\in V$ with $f_1({\nu^*})=0$ and
$\det(D\,f_1({\nu^*}))\neq0$, then there exists $\nu_{\e}$ such that
$g(\z_{\nu_{\e}},\e)=0$ and $\z_{\nu_{\e}}\to \z_{\nu^*}$ when
$\e\to0$.

\item[$(b)$] Assume that $f_1=0$.  If there exists ${\nu^*} \in V$
with $f_2({\nu^*})=0$ and $\det(D\,f_2({\nu^*})) \neq0$, then there
exists $\nu_{\e}$ such that $g(\z_{\nu_{\e}}, \e)=0$ and
$\z_{\nu_{\e}}\to \z_{\nu^*}$ when $\e\to0$.
\end{itemize}
\end{lemma}

The proof of this lemma can be found in \cite{LN}.

Note that in Lemma \ref{LS} the functions $g_i$ for $i=0,1,2$ which
appears in the expression of \eqref{fi} and \eqref{gamma} are the
ones of the function
\begin{equation}\label{x1}
g(z,\e)=g_0(z)+\e g_1(z)+ \e^2 g_2(z)+\CO(\e^3),
\end{equation}
instead of the functions which appear in \eqref{gi}.

\begin{proof}[Proof of Theorem \ref{BRt1}]
Let $\psi(\T,\z,\e)$ be a periodic solution of system \eqref{BRs1}
such that $\psi(0,\z,\e)=\z$. Similarly let $\psi^{\pm}(\T,\z,\e)$
be the solutions of the systems $\x'=F^{\pm}(\T,\x,\e)$ such that
$\psi^{\pm}(0,\z,\e)=\z$. So
\[
\psi(\T,\z,\e)=\left\{\begin{array}{l}
\psi^+(\T,\z,\e) \quad \textrm{if}\quad 0\leq\T\leq\phi,\vspace*{-0.15cm}\\
\psi^-(\T,\z,\e) \quad \textrm{if}\quad \phi\leq\T\leq T.
\end{array}\right.
\]
Since the vector field \eqref{BRs1} is $T$--periodic, it may also
read
\[
\psi(\T,\z,\e)=\left\{\begin{array}{l}
\psi^+(\T,\z,\e) \quad \textrm{if}\quad 0\leq\T\leq\phi, \vspace*{-0.15cm}\\
\psi^-(\T,\z,\e) \quad \textrm{if}\quad \phi-T\leq\T\leq 0.
\end{array}\right.
\]

Now, we consider the function
$g(\z,\e)=\psi^+(\phi,\z,\e)-\psi^-(\phi-T,\z,\e)$. It is easy to
see that the solution $\psi(\T,\z,\e)$ is $T$--periodic in $\T$ if
and only if $g(\z,\e)=0$. So, from hypothesis $(H)$ we have that
$g(\z_{\nu,\e})=0$ for every $\z_{\nu,\e}\in\CZ$.

Using Taylor series to expand the functions $\psi^{\pm}(\T,\z,\e)$
in powers of $\e$ we obtain
\begin{equation}\label{psi}
\psi^{\pm}(\T,\z,\e)=y_0^{\pm}(\T,\z)+\e
y_1^{\pm}(\T,\z)+\e^2\dfrac{y_2^{\pm}(\T,\z)}{2}+\CO(\e^2),
\end{equation}
where $y_i(\T,\z)$ is given in \eqref{yi}. We shall omit the
computations for obtaining \eqref{psi}, nevertheless they can be
found in \cite{LNT}. Therefore, $g(\z,\e)=g_0(\z)+\e g_1(\z)+\e^2
g_2(\z)+\CO(\e^2)$, where $g_i(\z)=y^+_i(\phi,\z)-y^-_i(\phi-T,\z)$
for $i=0,1,2$. Moreover
\[
D g_0(\z)=\dfrac{\p \f^+}{\p \z}(\phi,\z)-\dfrac{\p \f^-}{\p
\z}(\phi-T,\z)=Y^+(\phi,\z)-Y^-(\phi-T,\z).
\]
So, from hypothesis of  Theorem \ref{BRt1} we have that the matrix $Dg_0(\z)$ has in
the upper right corner the zero $q\times(d-q)$ matrix, and in the
lower right corner has the $(p-q)\times(p-q)$ matrix $\Delta_{\nu}$
with $\det(\Delta_{\nu})\neq0$.


We conclude the proof of this theorem by applying Lemma \ref{LS} to the function $g(\z,\e)$ defined in
\eqref{x1}.
\end{proof}

\section{Proof of Theorem \ref{th3}}\label{phineqpi}

In order to prove Theorem \ref{th3}  we shall study
the zeros of the averaging functions $f_1$ and $f_2$, given in
\eqref{f1} and \eqref{f2}, respectively, when $\phi \in
(0,2\pi)\setminus\{\pi\}$.

\begin{remark}
For sake of simplicity we shall denote  by $\lambda_{ijk_1 \ldots
k_m0}$ the coefficient of $x^i y^jz_1^{k_1} \ldots z_m^{k_m} $, and
by $\lambda_{ij0}$ the coefficient of $x^i y^j$ of system
\eqref{As1}, when $\lambda=a^\pm,b^\pm, c_\ell^\pm$ for all $1 \leq
\ell \leq m$.
\end{remark}

From statement $(a)$ of Lemma \ref{lemtrig} we have $f_1(\nu)=
(f_{10}(\nu),\ldots, f_{1m}(\nu))$ where
\begin{equation}\label{f1l}
\begin{array}{ll}
f_{10}(\nu) =&  \dis\sum_{i+j+k_1+\ldots+k_m=0}^{n} r^{i+j}z_1^{k_1}
\ldots z_m^{k_m} \bigg( a^+_{ijk_1\ldots k_m0 } I_{(i+1,j,\phi)} \\
& +  \,b^+_{ijk_1\ldots k_m 0} \dis I_{(i,j+1,\phi)}  + \,
a^-_{ijk_1\ldots k_m0 } J_{(i+1,j,\phi)} +
b^-_{ijk_1\ldots k_m 0}\dis  J_{(i,j+1,\phi)} \bigg),\\
f_{1\ell}(\nu) = & \dis \sum_{i+j+k_1+\ldots+k_m=0}^{n}
r^{i+j}z_1^{k_1} \ldots z_m^{k_m}  \bigg( c^+_{\ell,ijk_1\ldots k_m
0 } I_{(i,j,\phi)} + c^-_{\ell,ijk_1\ldots k_m 0} J_{(i,j,\phi)}
\bigg),
\end{array}
\end{equation}
with  $\nu=(r,z_1, \ldots, z_m)$ and $1 \leq \ell \leq m $.

\begin{proposition}
\label{prop10} Assume $0 \leq m \leq d$ and $\phi \neq \pi$. Then
$f_1$ has at most $n^{m+1}$ simple zeros and this number can be
reached.
\end{proposition}

\begin{proof}
For each $0 \leq \ell \leq m$ and $\nu=(r,z_1,\ldots,z_m)$,
$f_{1\ell}(\nu)$ is a complete polynomial of degree $n$.  Recall
that a {\it complete polynomial of degree $k$} means a polynomial
that appears all its monomials. By Bezout Theorem (see \cite{Ful}),
$f_1(\nu)$ can be at most $n^{m+1}$ simple zeros. Since all the
coefficients of $f_1(\nu)$ are independent, we can choose them in
order that $f_1(\nu)$ has exactly $n^{m+1}$ zeros with $r>0$, and
$\det f'_1(\nu^*)\neq 0$ for each zero $\nu^*$ of $f_1(\nu)$ (that
is, $\nu^*$ is a simple zero).
\end{proof}

\begin{proposition}\label{prop12}
Take $ 0 \leq m \leq d$ and $\phi \neq \pi$. If  $f_1 \equiv 0$ then
$f_2$ has at most $(2n)^{m+1}$ simple  zeros, and a lower bound for
the maximum number of simple zeros is $(2n)(2n-1)^{m}$.
\end{proposition}

\begin{proof}
Assume that $f_1 \equiv 0$. From \eqref{f1l} it follows that
\begin{equation}
\label{ac}
\begin{array}{ll}
\dis \sum_{i+j=s} & a^+_{ijk_1\ldots k_m0} \,I_{(i+1,j,\phi)} +
b^+_{ijk_1\ldots k_m 0} \,I_{(i,j+1,\phi)}  \\ & + \,a^-_{ijk_1
\ldots k_m0} \,J_{(i+1,j,\phi)}   +  b^-_{ijk_1\ldots k_m0} \,
J_{(i,j+1,\phi)}   =0,  \\
\dis\sum_{i+j=s} & c^+_{\ell,ijk_1\ldots k_m0} \,I_{(i,j,\phi)}+
\,c^-_{\ell,ijk_1\ldots k_m0} \,J_{(i,j,\phi)} =0,
\end{array}
\end{equation}
for   $1\leq \ell \leq m$,  $0 \leq s \leq n$, $0 \leq k_\ell \leq
n$ with $0 \leq k_1 + \ldots + k_m \leq n-s$.

Moreover, $f_2(\nu)=(f_{20}(\nu), \ldots, f_{2m}(\nu))$ with
$\nu=(r,z_1,\ldots,z_m)$. In particular, if $m=0$ then $f_2(\nu) =
f_{20}(r)$.  Considering the expression for $f_{2\ell}(\nu)$, given
in \eqref{f2l1}  for  $0 \leq \ell \leq m$, we conclude that
$\widetilde{G}_{1\ell}(\nu)$ and  $\dint_0^\phi
\widetilde{F}_{1\ell}^+(s,\z_\nu)\, ds +\dint_\phi^{2\pi}
\widetilde{F}_{1\ell}^-(s,\z_\nu)\, ds$ are complete polynomials of
degree $2n-1$ in the variables $(r, z_1,\ldots,z_m)$, and
$$
\int_0^\phi F^+_{2\ell}(s,\z_\nu)\,ds + \int_\phi^{2\pi}
F^-_{2\ell}(s,\z_\nu)\,ds=\frac{1}{r} \sum_{k=0}^{2n}Q_k (z_1,
\ldots, z_m)\,r^{k},
$$
where $\z_{\nu}=(r,z_1,\ldots,z_m,0,\ldots,0)\in\R^{d+1}$,  $Q_k
(z_1, \ldots, z_m)$ is a complete polynomial  of degree $2n-k$ in
the variables $(z_1, \ldots,z_m)$ if $m \neq 0$, and $Q_k (z_1,
\ldots, z_m)$ is constant if $m=0$. The above equality is evident if
we take into account statement $(a)$ of Lemma \ref{lemtrig} and
conditions \eqref{ac}. Therefore, each $r f_{2\ell}(\nu)$ is a
complete polynomial of degree $2n$ in the variables
$(r,z_1,\ldots,z_m)$.  Since $r>0$, it is known that $r
f_{2\ell}(\nu)=0$ if and only if $f_{2\ell}(\nu)=0$ for each $0 \leq
\ell \leq m$. Then, by Bezout Theorem,  $f_{2}(\nu)$ has at most
$(2n)^{m+1}$ simple zeros for all $0 \leq m \leq d$.

In order to show that the maximum number is greater than or equal to
$(2n)(2n-1)^m$ we provide a particular example. So, take
$a^\pm_{i00}\neq 0$, $c^\pm_{\ell,00\ldots 0 k_\ell 0} \neq 0$, and
we take zero all the other coefficients for  $ 1\leq \ell \leq m$.
From \eqref{f2l1} we obtain $f_{20}(\nu)=f_{20}(r)$ and
$f_{2\ell}(\nu)=f_{2\ell}(r,z_\ell)$, where
$$
\begin{array}{ll}
f_{20}(r) =&  \dis\frac{4}{r} \sum_{i=0}^{n} \sum_{p=0}^{n} r^{i+p}
\bigg( a^+_{i00}a^+_{p00} I_{(i+p+1,1,\phi)} + a^-_{i00}a^-_{p00} J_{(i+p+1,1,\phi)}  \\
&+\, i\,a^+_{i00} a^+_{p00} \dis \int_{0}^{\phi} \cos^{i+1}s
\,I_{(p+1,0,s)} ds  + \, i\,a^-_{i00} a^-_{p00} \dis
\int_{\phi}^{2\pi} \cos^{i+1} s \,I_{(p+1,0,s)}ds \bigg),
\end{array}
$$

$$
\begin{array}{ll}
f_{2\ell}(r,z_\ell)=& \dis\frac{4}{r} \sum_{i=0}^{n}\sum_{k_l=0}^{n}
r^i z_\ell^{k_\ell} \big( a^+_{i00} c^+_{\ell,0\ldots 0k_\ell 0}
I_{(i,1,\phi)} + a^-_{i00} c^-_{\ell,0\ldots0 k_\ell 0}  J_{(i,1,\phi)}\big)\\
&+ 4{\dis \sum_{ k_\ell=1}^{n} \,\sum_{ L_\ell=0}^{n}} z_\ell^{k_\ell+
L_\ell -1} \bigg(  \frac{\phi^2}{2} k_\ell\, c^+_{\ell,0\ldots 0k_\ell 0}
\,c^+_{\ell,0\ldots 0L_\ell 0}  \\
&+   \frac{(2\pi)^2-\phi^2}{2} k_\ell \,c^-_{\ell,0\ldots0 k_\ell0}
\, c^-_{\ell,0\ldots 0 L_\ell 0}  \bigg),
\end{array}
$$
where $a^+_{i00} I_{(i+1,0,\phi)}= -a^-_{i00}J_{(i+1,0,\phi)} $ and
$c^+_{\ell,00\ldots 0 k_\ell 0} I_{(0,0,\phi)} =-c^-_{\ell,00\ldots
0 k_\ell 0} J_{(0,0,\phi)}$ for $1 \leq \ell \leq m$ (see
\eqref{ac}).

{From} statement $(a)$ of Lemma \ref{lemtrig}, $r f_{20} (r)$ is a
complete polynomial of degree $2n$ in the variable $r$,  whose
coefficients are independent.  Furthermore, if  $f_{20}(r^*)=0$ with
$r^*>0$, then $f_{2\ell}(r^*,z_\ell)$ is a polynomial of degree
$2n-1$ in the variable  $z_\ell$,  and all their coefficients are
independent for $1 \leq \ell \leq m$. Therefore, By Bezout Theorem,
$f_2(\nu)$ has at most $(2n)(2n-1)^{m}$ simple zeros, and this
number can be reached due to the independence of coefficients.
\end{proof}

\begin{proof}[Proof of Theorem \ref{th3}]
We apply Theorem \ref{BRt1} to the function $f_1$ of Proposition
\ref{prop10} and we conclude statement $(a)$. Statement $(b)$ is
proved applying Theorem \ref{BRt1} to the functions $f_1$ and $f_2$
given in Proposition \ref{prop12}. 
\end{proof}

\subsection{Improving the lower bound}\label{improve}
As mentioned in the introduction, the lower bound of statement $(b)$ of Theorem \ref{th3} is not optimal and can be improved. From Theorem \ref{BRt1} we need to solve the equation $f_2(\nu)=0,$ assuming $f_1 \equiv 0$. This can be a hard task due to the complexity of $f_2$. In what follows, we provide a simpler polynomial system for which their simple zeros imply the existence of simple zeros of  $f_2$. 

 From \eqref{f2} we have $f_2(\nu)
=(f_{20}(\nu), \ldots, f_{2m}(\nu))$. In \eqref{f2l1} we can take
$\widetilde{G}_{1\ell}(\nu)=0$ and, since $1/r$ appears as a common factor in the expression of $A^\pm_{1}$ \eqref{AB}, we define $\dis  \widetilde{A}^\pm_{1}= r {A^\pm_1}.$ Finally,  for $1 \leq \ell \leq m,$ we assume that $ A^\pm_{\ell+2}=\delta \widetilde{A}^\pm_{\ell+2} $   and $ B^\pm_{\ell+2}=\delta \widetilde{B}^\pm_{\ell+2} $  for
$\delta>0$ sufficiently small. Notice that, the assumption is equivalent to ask that the coefficients of the perturbation \eqref{pert} for $1 \leq \ell \leq m$ are of order $\delta$.

Now, for $1 \leq \ell \leq m,$ we define
\begin{equation}\label{pq}
\begin{array}{ll}
P_{\ell} (\nu) =& \dis  \int_0^\phi   \widetilde{B}^+_{\ell+2}(s,\z_\nu)
\,ds + \int_\phi^{2\pi}  \widetilde{B}^-_{\ell+2}(s,\z_\nu) \,ds,\vspace*{0.2cm}\\
Q_\ell(\nu) =& \dis \int_0^\phi \widetilde{A}_1^+(s,\z_\nu)
\widetilde{A}^+_{\ell+2}(s,\z_\nu) \,ds + \dis \int_\phi ^{2\pi}
\widetilde{A}_1^-(s,\z_\nu) \widetilde{A}^-_{\ell+2}(s,\z_\nu) \,ds.
\end{array}
\end{equation}

Thus, from \eqref{AB}, \eqref{FF}, \eqref{yii} and \eqref{Ftil}  we have
$\dis \int \widetilde{F}_{1\ell}^\pm (s,\z_\nu) ds= \CO_2(\delta)$ and,
therefore,
\begin{equation*}
\begin{array}{ll}
\dis \frac{r}{4 \delta} f_{2\ell} (\nu)  = & \dis r P_{\ell}(\nu)  -
Q_\ell(\nu)    + \CO(\delta),\quad \text{for}\hspace*{0.2cm} 1 \leq  \ell \leq m.
\end{array}
\end{equation*}
Hence, taking $\delta>0$ sufficiently small, we obtain the following proposition.

\begin{proposition}
If the polynomial system
\begin{equation}\label{ps}
f_{20}(\nu)=0\,\, \text{ and }\,\, r\, P_{\ell}(\nu)-Q_{\ell}(\nu)=0,\,\, \text{ for }\,\, 1\leq\ell\leq m,
\end{equation}
has $N$ isolated solutions, then $N_2(m,n,\phi)\geq N$.
\end{proposition}

\section{Proof of Theorem \ref{th6}}\label{phi=pi}

In this section we  study the zeros of the functions $f_1$ and
$f_2$, given in \eqref{f1} and \eqref{f2}, respectively, when
$\phi=\pi$. Then, we conclude Theorem \ref{th6}  applying Theorem
\ref{BRs1}.

{From} statement $(b)$ of Lemma \ref{lemtrig} we have
$f_1(\nu)=(f_{10}(\nu),\ldots, f_{1\ell}(\nu))$ where
\begin{equation}\label{f11l}
\begin{array}{ll}
f_{10}(\nu) =&    \dis\sum_{i \,odd,\,P=0}^{n} r^{i+j}z_1^{k_1}\ldots
z_m^{k_m}     \bigg(    a^+_{ijk_1\ldots k_m0 } I_{(i+1,j,\pi)} +
a^-_{ijk_1\ldots k_m0 } J_{(i+1,j,\pi)} \bigg)\\
&+ \dis\sum_{i\,even,\,P=0}^{n} r^{i+j}z_1^{k_1}\ldots z_m^{k_m}
\bigg(  b^+_{ijk_1\ldots k_m 0} \dis I_{(i,j+1,\pi)}
+\, b^-_{ijk_1\ldots k_m 0}\dis  J_{(i,j+1,\pi)} \bigg),\\
f_{1\ell}(\nu) = & \dis \sum_{i \,even,\,P=0}^{n} r^{i+j}z_1^{k_1}
\ldots z_m^{k_m}   \bigg( c^+_{\ell,ijk_1\ldots k_m 0 }
I_{(i,j,\pi)} + c^-_{\ell,ijk_1\ldots k_m 0} J_{(i,j,\pi)} \bigg),
\end{array}
\end{equation}
where $\nu = (r,z_1, \ldots, z_m)$, $1 \leq \ell \leq m $ and $P=
i+j +k_1 + \ldots + k_m$.

\begin{proposition}\label{prop16}
Take $0 \leq m \leq d$ and $\phi =\pi$. Then, $f_1$ has at most
$n^{m+1}$ simple zeros and this number can be reached.
\end{proposition}

\begin{proof}
This proof is analogously to the proof of Proposition \ref{prop10},
noticing that for each $0 \leq \ell \leq m$, $f_{1\ell}(\nu)$ is a
complete polynomial of degree $n$ in the variables $(r, z_1, \ldots,
z_m)$ and all their coefficients are independent.
\end{proof}

\begin{proposition}\label{prop18}
Assume $ 0 \leq  m \leq d$ and $\phi =\pi$. If  $f_1 \equiv 0$ then
$f_2$ has at most $(2n)^{m+1}$ simple zeros, and the lower bound for
the number of simple zeros is $(2n-1)^{m+1}$ if $n$ is odd, and
$(2n-2)(2n-1)^{m}$ if $n$ is even.
\end{proposition}

\begin{proof}
Assume that $f_1 \equiv 0$. From \eqref{f11l} it follows that
\begin{equation}
\label{ac2}
\begin{array}{ll}
\dis \sum_{i\,odd,\,i+j=s} & a^+_{ijk_1\ldots k_m0} \,I_{(i+1,j,\pi)} +
\,a^-_{ijk_1\ldots k_m0} \,J_{(i+1,j,\pi)} \\
+ \dis \sum_{i\,even,\,i+j=s}  &  b^+_{ijk_1\ldots k_m 0} \,I_{(i,j+1,\pi)}
+  b^-_{ijk_1\ldots k_m0} \,J_{(i,j+1,\pi)} =0,  \\
\dis\sum_{i\,even,\,i+j=s} & c^+_{\ell,ijk_1\ldots k_m0}
\,I_{(i,j,\pi)}+ \,c^-_{\ell,ijk_1\ldots k_m0} \,J_{(i,j,\pi)} =0,
\end{array}
\end{equation}
for   $1\leq \ell \leq m$,  $\,0 \leq s \leq n$, $\,0 \leq k_\ell
\leq n$ with $0 \leq k_1 + \ldots + k_m \leq n-s$.

Moreover, $f_2(\nu)=(f_{20}(\nu), \ldots, f_{2m}(\nu))$ with $\nu
=(r,z_1, \ldots, z_m)$. If $m=0$ then $f_2(\nu)=f_{20}(r)$.
Analogously to the proof of Proposition \ref{prop12} we conclude
that $f_2(\nu)$ has at most $(2n)^{m+1}$ simple zeros for all $0
\leq m \leq d$.

Now, we provide a particular example to exhibit the lower bound for
the maximum number of simple zeros. So, take $a^\pm_{i00}\neq 0$,
$c^\pm_{\ell,00\ldots 0 k_\ell 0} \neq 0$, and take zero all the
other coefficients for  $ 1\leq \ell \leq m$. From \eqref{f2l1} we
obtain $f_{20}(\nu)=f_{20}(r)$ and
$f_{2\ell}(\nu)=f_{2\ell}(r,z_\ell)$, where
$$
\begin{array}{ll}
f_{20}(r) =& \dis\frac{4}{r} \left(   \dis \sum_{i\,even,\,i=0}^{n}
\,\sum_{p\,odd,\,p=0}^{n} r^{i+p} \big( a^+_{i00}a^+_{p00}
I_{(i+p+1,1,\pi)}
+\, a^-_{i00}a^-_{p00} J_{(i+p+1,1,\pi)} \right. \vspace*{0.15cm}\\
& +\, i\,a^+_{i00} a^+_{p00} \dis \int_{0}^{\pi} \cos^{i+1}s \,I_{(p+1,0,s)} ds
+\, i\,a^-_{i00} a^-_{p00} \dis \int_{\pi}^{2\pi} \cos^{i+1} s \,I_{(p+1,0,s)}ds
\big) \vspace*{0.15cm}\\
&+ \dis\sum_{i \,odd,\,i=0}^{n} \,\sum_{p\,even,\,p=0}^{n} r^{i+p}
\big( a^+_{i00}a^+_{p00} I_{(i+p+1,1,\pi)} +\, a^-_{i00}a^-_{p00}
J_{(i+p+1,1,\pi)}\, \vspace*{0.15cm}\\
&+\, i\,a^+_{i00} a^+_{p00} \dis \int_{0}^{\pi} \cos^{i+1}s \,I_{(p+1,0,s)} ds
+\, i\,a^-_{i00} a^-_{p00} \dis \int_{\pi}^{2\pi} \cos^{i+1} s \,I_{(p+1,0,s)}ds
\big) \vspace*{0.15cm}\\
&+ \dis \sum_{i\,odd,\,i=0}^{n} \,\sum_{p\,odd,\,p=0}^{n} r^{i+p}
\big( a^+_{i00}a^+_{p00} I_{(i+p+1,1,\pi)} +\, a^-_{i00}a^-_{p00}
J_{(i+p+1,1,\pi)} \,\vspace*{0.15cm}\\
&  \left. + i\,a^+_{i00} a^+_{p00} \dis \int_{0}^{\pi} \cos^{i+1}s
\,I_{(p+1,0,s)} ds   + i a^-_{i00} a^-_{p00} \dis \int_{\pi}^{2\pi}
\cos^{i+1} s I_{(p+1,0,s)}ds \big) \right),
\end{array}
$$
$$
\begin{array}{ll}
f_{2\ell}(r,z_\ell)=& \dis\dfrac{4}{r} \sum_{i=0}^{n}\sum_{k_l=0}^{n}
r^i z_\ell^{k_\ell} \big( a^+_{i00} c^+_{\ell,0\ldots 0k_\ell 0}
I_{(i,1,\phi)} + a^-_{i00} c^-_{\ell,0\ldots0 k_\ell 0}  J_{(i,1,\phi)}
\big) \vspace*{0.15cm}\\
&+{\dis \sum_{ k_\ell=1}^{n} \sum_{ L_\ell=0}^{n}}
z_\ell^{k_\ell+L_\ell -1} k_\ell\bigg(  \frac{\phi^2}{2}
c^+_{\ell,0\ldots 0k_\ell 0} c^+_{\ell,0\ldots 0L_\ell 0} +
\frac{(2\pi)^2-\phi^2}{2} c^-_{\ell,0\ldots0 k_\ell0 }
c^-_{\ell,0\ldots 0 L_\ell 0}  \bigg),
\end{array}
$$
for $1 \leq \ell \leq m$, where $a^+_{i00} I_{(i+1,0,\pi)}=
-a^-_{i00}J_{(i+1,0,\pi)}$ if $i$ is odd  and $c^+_{\ell,00\ldots 0
k_\ell 0} I_{(0,0,\pi)} =-c^-_{\ell,00\ldots 0 k_\ell 0}
J_{(0,0,\pi)}$ (see \eqref{ac2}). Therefore, from statement $(b)$ of
Lemma \ref{lemtrig},  $r f_{20}(r)$ is a complete polynomial in the
variable $r$ of degree $2n-1$ if $n$ is odd, and $2n-2$ if $n$ is
even, and its coefficients are independent.  Furthermore, if
$f_{20}(r^*)=0$ with $r^*>0$,  then $f_{2\ell}(r^*,z_\ell)$ is a
polynomial of degree $2n-1$ in  the variable $z_\ell$ for each $1
\leq \ell \leq m$. Then, the number of simple zeros with $r>0$ of
$f_2(\nu)$ can be $(2n-1)^{m+1}$ if $n$ is odd, and
$(2n-2)(2n-1)^{m}$ if $n$ is even. By the independence of all
coefficients these numbers can be reached.
\end{proof}

\begin{proof}[Proof of Theorem \ref{th6}]
From Theorem \ref{BRt1} and Proposition \ref{prop16},  statement
$(a)$  holds, and applying Theorem \ref{BRt1} to the functions $f_1$
and $f_2$ given in Proposition \ref{prop18} we conclude statement
$(b)$.
\end{proof}

\section{Proof of Theorem \ref{th7}}\label{phi=2pi}

When $\phi=2\pi$ system \eqref{As1} is continuous. Then, considering
the cylindrical coordinates given in \eqref{cyl}, and taking $\T$ as
the new time, system \eqref{As1} can be written as system
\eqref{BRs1} that is,
$$
{\bf x}' = F^+(\T,{\bf x},\e), \quad \text{for}\, 0 \leq \T \leq
2\pi,
$$
where
$$
F^+(\T,{\bf x},\e) = F_0^+(\T,{\bf x}) + \e F_1^+(\T,{\bf x}) + \e^2
F_2^+(\T,{\bf x}) + \e^3 R^+ (\T,{\bf x},\e),
$$
for ${\bf x}=(r,{\rm z})$ and ${\rm z}=(z_1, \ldots, z_d)$, with
$F_j^+$  given in \eqref{F} and \eqref{FF} for $j=0,1,2$.

{From} statement $(c)$ of Lemma \ref{lemtrig}   we have
$f_1(\nu)=(f_{10}(\nu),  \ldots, f_{1m}(\nu)$ with
\begin{equation}
\label{f1l.c}
\begin{array}{lll}
f_{10}(\nu) &=&  \dis\sum_{ i \,odd,\,j\,even,\,P=0}^{n} r^{i+j}z_1^{k_1}
\ldots z_m^{k_m} \, a^+_{ijk_1\ldots k_m 0} I_{(i+1,j,2\pi)} \\
&& +\dis\sum_{ i \,even,\,j\,odd,\,P=0}^{n} r^{i+j}z_1^{k_1}\ldots z_m^{k_m}
\, b^+_{ijk_1\ldots k_m0 } \dis I_{(i,j+1,2\pi)},\\
f_{1\ell}(\nu) &= & \dis \sum_{i,\,j \,even,\,P=0}^{n}
r^{i+j}z_1^{k_1} \ldots z_m^{k_m} \, c^+_{\ell,ijk_1\ldots k_m 0}
I_{(i,j,2\pi )} ,
\end{array}
\end{equation}
where $\nu=(r,z_1, \ldots, z_m)$,  $1\leq \ell \leq m$  and
$P=i+j+k_1+\ldots+k_m$.

\begin{proposition}\label{prop20}
Assume $0 \leq m \leq d$ and $\phi =2 \pi$. If $m \neq 0$ then $f_1$
has at most $n^{m}(n-1)/2$ simple zeros and this number can be reached.
If $m =0$ then $f_1$ has at most $(n-1)/2$ simple zeros if $n$ is
odd, and $(n-2)/2$ if $n$ is even, and these numbers can be reached.
\end{proposition}

\begin{proof}
We have  $f_{10}(\nu) = r \widetilde{f}_{10}(\nu)$ with
$$
\widetilde{f}_{10}(\nu) = h_1  +r^2 h_3+ r^4 h_5 +r^6 h_7+\ldots  +
\left\{\begin{array}{l}
 r^{n-1} h_{n} \quad \quad {\rm if ~n ~is ~odd},\vspace*{-0.3cm}\\
r^{n-2} h_{n-1}\quad {\rm if ~n ~is ~even},
\end{array} \right.
$$
where
$$
\begin{array}{ll}
h_k =& \dis \sum_{k_1 + \ldots + k_m=0}^{n-k} z_1^{k_1} \ldots z_m^{k_m}
\bigg( \sum_{i\,odd,\,j\,even,\, i+j=k}  a^+_{ijk_1 \ldots k_m0} \,I_{(i+1,j,2\pi)} \\
&\dis + \sum_{i\,even,\,j\,odd,\,i+j=k} b^+_{ijk_1 \ldots k_m0}
\,I_{(i,j+1,2\pi)} \bigg).
\end{array}
$$

If $m\neq 0$ then  $\widetilde{f}_{10}(\nu)$ and $f_{1\ell}(\nu)$
are polynomials in the variables $(r, z_1, \ldots, z_{m})$ of degree
$n-1$ and $n$, respectively, for $ 1 \leq \ell \leq m$. From Bezout
Theorem the maximum number of simple zeros  of $f_1(\nu)$ is
$n^m(n-1)$. Since the exponents of $r$ in the function
$\widetilde{f}_{10}(\nu)$ are always even numbers,  the maximum
number of simple zeros of $f_1(\nu)$ is $n^m(n-1)/2$. In what
follows we provide a  particular example to prove that this number
is reached.

First if $n$ is even we take  $a^+_{10k_10}\neq0$, $b^+_{01k_10}\neq
0$, $c^+_{1,ij0}\neq 0$, $c^+_{\ell, 00k_{\ell}0}   \neq 0$ and
take zero all the other coefficients in other that
$\widetilde{f}_{10}(\nu)=\widetilde{f}_{10}(z_1)$, $f_{11}(\nu
)=f_{11}(r )$, and $f_{1\ell}(\nu)=f_{11}(z_\ell)$, where
\begin{align*}
\widetilde{f}_{10}(z_1)=&\sum_{k_1=0}^{n-1} z_1^{k_1} \left(
a^+_{10k_10} I_{(2,0,2\pi)} +   b^+_{01k_10} I_{(0,2,2\pi)}
\right),\\
f_{11}(r ) =& \sum_{i,j\,even,\, i+j=0
}^{n} r^{i+j} \,c^+_{1,ij0}  \,I_{(i,j,2\pi)},\\
f_{1\ell}(z_\ell) =&\sum_{k_{\ell}=0}^{n} z_\ell^{k_\ell}\,
\,c^+_{\ell, 00k_{\ell}0}\, I_{(0,0,2\pi)},
\end{align*}
for $2\leq \ell \leq m$. Thus, $\widetilde{f}_{10}(z_1)$ is a
complete polynomial of degree $n-1$ in the variable $z_1$,
$f_{11}(r)$ is an even polynomial of degree $n$  in the variable
$r$, and $f_{1\ell}(z_\ell)$ is a complete polynomial of degree $n$
in the variable $z_\ell$ for all $2 \leq \ell \leq m$. Since  the
exponents of  $r$ in $f_{11}(r)$ is even, then $f_1(\nu)$ can have
$n^m(n-1)/2 $ simple zeros with $r>0$.

On the other hand, if $n$ is odd we take $a^+_{ij0}\neq0$,
$b^+_{ij0} \neq0$, $c^+_{\ell,00k_\ell0}\neq 0$ and we  take zero
all the other coefficients and then we obtain
$\widetilde{f}_{10}(\nu)=\widetilde{f}_{10}(r)$ and
$f_{1\ell}(\nu)=f_{1\ell}(z_\ell)$, where
$$
\begin{array}{ll}
\widetilde{f}_{10}(r) =& h_1 + r h_2+ r^2 h_3  + \ldots + r^{n-1} h_{n},\\
f_{1\ell}(\nu)= &\dis \sum_{k_{\ell}=0}^{n} {z_{\ell}}^{k_{\ell}} \,
c_{{\ell}, 00 k_{\ell}0}\, I_{(0,0,2\pi)},
\end{array}
$$
for $1 \leq \ell \leq m$. Then, $\widetilde{f}_{10}(r)$ is a
polynomial of degree $n-1$ in the variable $r$, whose exponents are
always even.  In a similar way $f_{1\ell}(z_\ell)$ is a polynomial
of degree $n$ in the variable $z_\ell$ for $1\leq \ell\leq m$.
Therefore, $f_1(\nu)$ can have $n^m(n-1)/2$ simple  zeros with
$r>0$.

If  $m=0$ then $\nu=r$ and $f_1(\nu) = r\widetilde{f}_{10}(r)$. So
the number of simple zeros can be $n-1$ if $n$ is odd, and $n-2$ if
$n$ is even. Since the exponent of $r$ in $\widetilde{f}_{10}$ is
even, the maximum number of simple zeros with $r>0$ of $f_1(\nu)$ is
$(n-1)/2$ if $n$ is odd, and $(n-2)/2$ if $n$ is even.

Now, we exhibit a particular example where the maximum number of
simple zeros of $f_1(\nu)$ can be reached. Take $a^+_{ij0}\neq0$,
$b^+_{ij0} \neq 0$ and we take zero all the other coefficients so
that $\widetilde{f}_{10}(r)$ is an even polynomial in  the variable
$r$ of degree $n-1$ if $n$ is odd, and $n-2$ is $n$ if even. So, the
number of simple zeros of $f_1(\nu)$ with $r>0$ can be $(n-1)/2$ if
$n$ is odd, and $(n-2)/2$ if $n$ is even.

In both particular cases, $m\neq 0$ and $m=0$, the coefficients of
$f_1(\nu)$ are independent. Therefore, the  maximum number of simple
zeros with $r>0$ of $f_1(\nu)$  can be reached.
\end{proof}

Now, we emphasize that the averaging function $f_2$ of the continuous
system \eqref{As1}, for $\phi=2\pi$, is given by $f_2(\nu)= \big(
f_{20}(\nu), \ldots, f_{2m}(\nu)  \big)$ being
\begin{equation}
\label{f2lc}
\begin{array}{ll}
f_{2\ell}(\nu)=&   2\dis\widetilde{G}_{1\ell}(\nu) +
4\int_{0}^{2\pi}  \big( F^+_{2\ell}(s,\z_\nu)  +
\widetilde{F}^+_{1\ell}(s,\z_\nu) \big)\,ds, \vspace*{0.2cm}
\end{array}
\end{equation}
for $ 0 \leq \ell \leq m$,  $F_{2\ell}^+$, $\widetilde{G}_{1\ell}$
and $\widetilde{F}_{1\ell}^+$  given in  \eqref{FF}, \eqref{G1l} and
\eqref{Ftil}, respectively.

\begin{proposition}\label{prop21}
Assume $m=0$ and $\phi=2\pi$. If $f_1\equiv 0$ then $f_2$ has at
most $2n$ simple zeros. Moreover, the lower bound for the number of simple
zeros is $n$.
\end{proposition}

\begin{proof}
If  $m=0$ then $\nu=r$ and $f_1(\nu)=f_{10}(r)$. Assume that $f_1
\equiv 0$. From \eqref{f1l.c} we obtain
\begin{equation}
\label{ab=0}
\begin{array}{l}
\dis\sum_{ i \,odd,\,j\,even, \,P =s}^{n} a^+_{ij0} I_{(i+1,j,2\pi)}
+\dis\sum_{ i \,even,\,j\,odd,\, P=s }^{n}  b^+_{ij0 } \dis
I_{(i,j+1,2\pi)}=0,
\end{array}
\end{equation}
where $P=i+j$ and $0 \leq s\leq n$.

Furthermore, by \eqref{f2lc} we have $f_2(\nu)=f_{20}(r)$. Therefore,
from statement $(c)$ of Lemma \ref{lemtrig} and \eqref{ab=0}, we
conclude that $\widetilde{G}_{10}(\nu)$ and $\dint_0^{2\pi}
\widetilde{F}_{10}(s,\z_{\nu})\,ds$  are complete polynomials of
degree $2n-1$ in the variable $r$, and
\begin{equation*}
\int_0^{2\pi} F^+_{20}(s,\z_{\nu}) \,ds = \sum_{s=0}^{N_1} R_s \,
r^{2 s+1}  +\frac{1}{r}\sum_{k=0}^{n} Q_k \, r^{2k},
\end{equation*}
where $\z_{\nu}=(r,0,\ldots,0)\in\R^{d+1}$, $R_s$ and $Q_k$ are constants,  $N_1= \dfrac{n-2}{2}$ if $n$ is even, and $N_1 = \dfrac{n-1}{2}$ if $n$ is odd. Therefore, $r
\dint_0^{2\pi} F^+_{20}(s,\z_{\nu}) \,ds$ is an even polynomial in
the variable $r$. Since $r>0$ it follows that  $r f_2(\nu)=0$ if and
only if $f_2(\nu)=0$. By Bezout Theorem the maximum number of simple
zeros of $f_2(\nu)$ is $2n$.

In order to exhibit the lower bound for the number of simple zeros
of $f_2(\nu)$, we provide a particular example.  Then, take
$a^\pm_{ij0} \neq 0$,  $b^\pm_{ij0} \neq 0$, $\alpha^\pm_{ij0} \neq
0$, and $\beta^\pm_{ij0} \neq 0$ and we take zero all the other
coefficients in such a way that  $f_{20}(r) =  4
\dint_0^{2\pi}F^+_{20}(r,\z_{\nu}) d\theta$. Therefore, $r f_{20}
(r)$ is a polynomial in $r$ of degree $2n$. Since $r f_{20} (r)$ is an even polynomial in
$r$, then the number of simple zeros of $f_2(\nu)$ with $r>0$ can be
$n$, and this number can
be reached due to the independence of all coefficients.
\end{proof}

\begin{proof}[Proof of Theorem \ref{th7}.]
Applying Theorem \ref{BRs1} to the function $f_1$ given in
Proposition \ref{prop20}, statement $(a)$ holds. We apply Theorem
\ref{BRs1} to the function $f_2$ given in Proposition \ref{prop21}
and we conclude statement $(b)$.
\end{proof}

\section*{Appendix}

In this appendix we shall exhibit some general expression of
functions that appears in subSection \ref{sub3.2}.

We will denote by $\lambda_{ijk_1 \ldots k_m 1_\omega}$  the
coefficient of $x^i y^j z_1^{k_1} \ldots z_m^{k_m} z_\omega$, and by
$\lambda_{ijk_1 \ldots k_m 0}$ the coefficient of $x^i y^j z_1^{k_1}
\ldots z_m^{k_m}$ in system \eqref{As1} when  $\lambda = a^\pm,
b^\pm, \alpha^\pm, \beta^\pm, c^{\pm}_{\ell}, \gamma^{\pm}_{\ell}$
for all  $ 0 \leq \ell \leq m$ and $m+1 \leq \omega \leq d$. Recall
that $\nu=(r,z_1,\ldots,z_m)$ and
$\z_{\nu}=(r,z_1,\ldots,z_m,0,\ldots,0)\in \R^{d+1}$.

For the next expressions, take $P=    i+j+k_1+\ldots +k_m$ and $Q=p+q+ L_1 + \ldots + L_m$.

{From} \eqref{AB} and \eqref{gl} we obtain
$$
\begin{array}{ll}
\dis\frac{\partial g_{10}}{\partial z_\l }(\z_\nu) =&\dis
\sum_{P=0}^{n-1} r^{i+j}z_1^{k_1} \ldots z_m^{k_m} \dis \bigg(
a^+_{ijk_1\ldots k_m 1_\l} \int_0^\phi e^{\mu_\omega s}\cos^{i+1} s\,
\sin^j s\,ds \vspace*{0.15cm}\\
&+ \,b^+_{ijk_1 \ldots k_m 1_\l} \, \dint_0^\phi e^{\mu_\omega s}\cos^i
s\,\sin^{j+1} s\,ds \vspace*{0.15cm}\\
& \dis + \,a^-_{ijk_1\ldots k_m 1_\l}  \int_\phi^{2\pi} e^{\mu_\omega s}
\cos^{i+1} s\,\sin^j s\,ds  \vspace*{0.15cm}\\
&+\, b^-_{ijk_1 \ldots k_m 1_\l} \, \dint_\phi^{2\pi} e^{\mu_\omega s}
\cos^{i} s\,\sin^{j+1} s\,ds  \bigg), \vspace*{0.15cm}\\
\dis\frac{\partial g_{1\ell}}{\partial z_\l}(\z_\nu) =&
\dis\sum_{P=0}^{n-1} r^{i+j}z_1^{k_1} \ldots z_m^{k_m}  \dis \bigg(
c^+_{l,ijk_1\ldots k_m 1_\l}
\int_0^\phi e^{\mu_\omega s} {\cos^i s}\, {\sin^j s} \, ds  \vspace*{0.15cm}\\
&+  \,c^-_{\ell,ijk_1\ldots k_m 1_\l} \dint_\phi^{2\pi}
e^{\mu_\omega s} {\cos^i s}\, {\sin^j s}\, ds \bigg),
\end{array}
$$
for  $1 \leq \ell \leq  m$ and $m+1 \leq \l \leq  d$.

{From} \eqref{gammaw}  we get
$$
\begin{array}{ll}
\gamma_\l(\nu) =& \dis \dfrac{-1}{1- e^{-\mu_{\omega}2\pi}} \sum_{P=0}^{n}
r^{i+j} z_1^{k_1} \ldots z_m^{k_m} \bigg(     c^+_{\l,ijk_1\ldots k_m0}
\, \dis \int_0^{\phi} e^{-\mu_\omega s} \cos^i s\,\sin^j s\,ds \vspace*{0.15cm} \\
&+ ~ c^-_{\l,ijk_1\ldots k_m0} \, \dis\int_{\phi}^{2\pi}
e^{-\mu_\omega(2\pi +s)} \cos^i s\,\sin^js\,ds \bigg),
\end{array}
$$
for $ m+1 \leq \l \leq d$.

{From} the above equalities and \eqref{G1l} we obtain for $1 \leq
\ell \leq m$ that
$$
\begin{array}{l}
\widetilde{G}_{10}(\nu) = \dis \sum_{\l=m+1}^{d} \Bigg[
\sum_{P=0}^{n-1} r^{i+j}z_1^{k_1} \ldots z_m^{k_m} \bigg(
a^+_{ijk_1\ldots k_m 1_\l}
\dint_0^\phi e^{\mu_\omega s}\cos^{i+1} s\,\sin^j s\,ds \vspace*{0.15cm}\\
+ ~b^+_{ijk_1 \ldots k_m 1_\l} \, \dint_0^\phi e^{\mu_\omega s}\cos^i
s\,\sin^{j+1} s\,ds  \dis + \,a^-_{ijk_1\ldots k_m 1_\l}  \dint_\phi^{2\pi}
e^{\mu_\omega s}\cos^{i+1} s\,\sin^j s\,ds \vspace*{0.15cm}\\
+ ~ b^-_{ijk_1 \ldots k_m 1_\l} \, \dint_\phi^{2\pi} e^{\mu_\omega
s}\cos^{i} s\,\sin^{j+1} s\,ds  \bigg) \Bigg]  \Bigg[
\dis\dfrac{-1}{1- e^{-\mu_{\omega}2\pi}} \sum_{Q=0}^{n} r^{i+j} z_1^{k_1}
\ldots z_m^{k_m}  \vspace*{0.15cm}\\
\bigg(c^+_{\l,ijL_1\ldots L_m0}  \, \dis \int_0^{\phi} e^{-\mu_\omega s} \cos^i
s\,\sin^js\,ds + \, c^-_{\l,ijL_1\ldots L_m0} \,
\dis\dint_{\phi}^{2\pi} \,e^{-\mu_\omega(2\pi+s)} \cos^i s\,\sin^js\,ds \bigg)
\Bigg],
\end{array}
$$
$$
\begin{array}{l}
\widetilde{G}_{1\ell}(\nu) = \dis \sum_{\l=m+1}^{d} \Bigg[
\sum_{P=0}^{n-1} r^{i+j}z_1^{k_1} \ldots z_m^{k_m}  \dis \bigg(
c^+_{\ell,ijk_1\ldots k_m 1_\l}
\int_0^\phi e^{\mu_\omega s} {\cos^i s}\, {\sin^j s} \, ds  \vspace*{0.15cm}\\
+  \,c^-_{\ell,ijk_1\ldots k_m 1_\l} \dint_\phi^{2\pi} e^{\mu_\omega
s} {\cos^i s}\, {\sin^j s}\, ds \bigg) \Bigg]  \Bigg[
\dis\frac{-1}{1- e^{-\mu_{\omega}2\pi}} \sum_{Q=0}^{n} r^{i+j} z_1^{L_1}
\ldots z_m^{L_m}  \vspace*{0.15cm}\\
\bigg(     c^+_{\l,ijL_1\ldots L_m0}  \, \dis \int_0^{\phi} e^{-\mu_\omega s}
\cos^i s\,\sin^js\,ds 
+  \,c^-_{\l,ijL_1\ldots L_m0}  \,
\dis\int_{\phi}^{2\pi} e^{-\mu_\omega(2\pi +s)} \cos^i s\,\sin^js\,ds \bigg) \Bigg].
\end{array}
$$

Now, from \eqref{FF} we compute
\begin{align*}
\dis\frac{\p F^\pm_{10}}{\p r} (s,\varphi(s,\z_\nu))= &
\frac{1}{r} \sum_{P=0}^{n}  (i+j) \,r^{i+j} z_1^{k_1} \ldots z_m^{k_m} \\
&\bigg( a^\pm_{ijk_1...k_m0}{\cos^{i+1} s} \,{\sin^j  s}+
b^\pm_{ijk_1...k_m0}
{\cos^i s} \,{\sin^{j+1} s} \bigg),    \\
\dis\frac{\p F^\pm_{10}}{\p z_\rho} (s,\varphi(s,\z_\nu))= &
\sum_{P=0}^{n}  k_\rho\, r^{i+j} z_1^{k_1} \ldots z_\rho^{k_\rho -1}
\ldots z_m^{k_m} \\
&\bigg( a^\pm_{ijk_1...k_m0} {\cos^{i+1} s }\,{\sin^j s} + b^\pm_{
ijk_1...k_m0} {\cos^{i} s }\,{\sin^{j+1}  s}  \bigg) ,\\
\dis\frac{\p F^\pm_{10}}{\p z_\l} (s,\varphi(s,\z_\nu))= &
\sum_{P=0}^{n-1}  r^{i+j} z_1^{k_1} \ldots z_m^{k_m} \\
&\bigg( a^\pm_{ijk_1...k_m 1_\l} {\cos^{i+1} s }\,{\sin^j s} +
b^\pm_{ijk_1...k_m1_\l} {\cos^{i} s}\,{\sin^{j+1}s}   \bigg),
\end{align*}
\begin{align*}
\dis \frac{\p F_{1\ell}^\pm}{\p r}(s,\varphi(s,\z_\nu)) =&
\frac{1}{r} \sum_{P = 0}^{n} (i+j) \,
r^{i+j}  z_1^{k_1} \ldots z_m^{k_m} \, c^\pm_{\ell, ijk_1...k_m0}
{\cos^i s}\, {\sin^j s} , \\
\dis \frac{\p F_{1\ell}^\pm}{\p z_\rho}(s,\varphi(s,\z_\nu)) =
&\sum_{P = 0}^{n} \,k_\rho\,  r^{i+j} z_1^{k_1} \ldots z_p^{k_\rho -1}
\ldots z_m^{k_m} \,c^\pm_{\ell, ijk_1...k_m0} {\cos^i  s}\,{\sin^j s} ,\\
\dis \frac{\p F_{1\ell}^\pm}{\p z_\l}(s,\varphi(s,\z_\nu)) = &
\sum_{P=0}^{n-1}  r^{i+j} z_1^{k_1}  \ldots  z_m^{k_m} \,
c^\pm_{\ell,ijk_1...k_m 1_\l} {\cos^i s} \,{\sin^j s} ,
\end{align*}
for  $1 \leq \ell \leq m$, $1 \leq \rho \leq  m$ and $m+1 \leq \l
\leq d$.

Note that when $m=d$  we do not consider the functions $\dis
\frac{\p F_{10}^\pm}{\p z_\l}$ and $\dis \frac{\p F_{1\ell}^\pm}{\p
z_\l}$.

Now, from \eqref{AB} and \eqref{yii}  we get
\begin{equation*}
\begin{array}{lll}
y^\pm_{10}(s,\z_\nu) & = & \dis \sum_{P=0}^n r^{i+j}  z_1^{k_1}
\ldots z_m^{k_m} \,\big( a^\pm_{ijk_1\ldots k_m 0} I_{(i+1,j,s)}  +
\,b^\pm_{ijk_1\ldots k_m 0}
I_{(i,j+1,s)} \big),\\
y^\pm_{1\rho}(\T,\z_\nu) & = & \dis \sum_{P=0}^{n} \,r^{i+j} z_1^{k_1}
\ldots z_m^{k_m} \,c^\pm_{\rho,ijk_1 \ldots k_m 0} I_{(i,j,s)}, \\
y^\pm_{1\omega}(\T,\z_\nu) & = & \dis  \sum_{P=0}^{n} \, r^{i+j}
z_1^{k_1} \ldots z_m^{k_m} \,c^\pm_{\l,ijk_1 \ldots k_m 0} \,
\int_{0}^{s}  e^{\mu_\l(s-\tau)}  \cos^i \tau\,\sin^js\,d\tau ,
\end{array}
\end{equation*}
for $ 1 \leq \rho \leq m$ and $ m+1 \leq \omega \leq d$. Therefore
$$
\begin{array}{l}
\dis\int \dis\frac{\p F^\pm_{10}}{\p r} (s,\varphi(s,\z_\nu))
y^\pm_{10}(s,\z_\nu)\,ds= \dis \dfrac{1}{r}
\sum_{P=0}^{n}\,\sum_{Q=0}^{n}\,
(i+j)\, r^{i+j+p+q} \,z_1^{k_1+L_1} \ldots z_m^{k_m+L_m}  \vspace*{0.15cm}\\
\dis \bigg( a^\pm_{ijk_1...k_m0} a^\pm_{pqL_1\ldots L_m 0} \int
{\cos^{i+1} s} \,{\sin^j  s}\,
I_{(p+1,q,s)}ds\vspace*{0.15cm}\\
\dis +\, b^\pm_{ijk_1...k_m0} a^\pm_{pqL_1\ldots L_m 0} \int {\cos^i
s} \,{\sin^{j+1} s} \,
I_{(p+1,q,s)}ds\vspace*{0.15cm}\\
\dis + \,  a^\pm_{ijk_1...k_m0} b^\pm_{pqL_1\ldots L_m 0}\int {\cos^{i+1} s}
\,{\sin^j  s}\,I_{(p,q+1,s)}ds \vspace*{0.15cm} \\
\dis +\, b^\pm_{ijk_1...k_m0}  b^\pm_{pqL_1\ldots L_m 0}  \int
\cos^is\,\sin^{j+1}s\, I_{(p,q+1,s)}ds \bigg) ,
\end{array}
$$
$$
\begin{array}{l}
\dis \int \frac{\p F^\pm_{10}}{\p z_{\rho}} (s,\varphi(s,\z_\nu))
y^\pm_{1\rho}(s,\z_\nu) ds= \dis \sum_{P=0}^{n}\sum_{Q=0}^{n}
k_\rho r^{i+j+p+q} z_1^{k_1+L_1} \ldots z_\rho^{k_\rho + L_\rho-1}
\ldots z_m^{k_m+L_m} \vspace*{0.15cm}\\
\dis \bigg( a^\pm_{ijk_1...k_m0} \,c^\pm_{\rho,pqL_1 \ldots L_m 0}
\int {\cos^{i+1} s }\,{\sin^j s} \,I_{(p,q,s)} \,ds \vspace*{0.15cm}\\
+\, b^\pm_{ijk_1...k_m0} \, c^\pm_{\rho,pqL_1 \ldots L_m 0} \int
{\cos^{i} s }\,{\sin^{j+1}s\,I_{(p,q,s)}} \,ds \bigg)  ,
\end{array}
$$
$$
\begin{array}{l}
\dis \int \frac{\p F^\pm_{10}}{\p z_{\omega}} (s,\varphi(s,\z_\nu))
y^\pm_{1\omega}(s,\z_\nu) ds=
\dis \sum_{P=0}^{n-1} \,\sum_{Q=0}^{n} r^{i+j+p+q} z_1^{k_1+L_1}\,
z_2^{k_2+L_2} z_3^{k_3+L_3} ... \, z_m^{k_m+L_m} \vspace*{0.15cm}\\
\dis \Bigg( a^\pm_{ijk_1...k_m1_\l} c^\pm_{\l,pqL_1 \ldots L_m 0}
\int {\cos^{i+1} s }\,{\sin^j s} \,  \bigg( \int_0^s e^{\mu_\l(s-\tau)}
\cos^p\tau \sin^q \tau d\tau \bigg) ds \vspace*{0.15cm} \\
\dis +\,b^\pm_{ijk_1...k_m1_\l} c^\pm_{\l,pqL_1 \ldots L_m 0} \int
{\cos^{i} s}\,{\sin^{j+1}s}  \, \bigg( \int_0^s e^{\mu_\l(s-\tau)} \cos^p\tau
\sin^q \tau d\tau \bigg) ds \Bigg) ,
\end{array}
$$

$$
\begin{array}{l}
\dis \int \frac{\p F_{1\ell}^\pm}{\p r}(s,\varphi(s,\z_\nu))
y^\pm_{10}(s,\z_\nu) ds= \dis \frac{1}{r} \sum_{P = 0}^{n} \sum_{Q =
0}^{n} (i+j) \,
r^{i+j+p+q}  z_1^{k_1+L_1} z_2^{k_2+L_2} ... \, z_m^{k_m+L_m} \vspace*{0.15cm}\\
\dis \bigg( c^\pm_{\ell, ijk_1...k_m0} a^\pm_{pqL_1\ldots L_m 0}
\int  {\cos^i s}\, {\sin^j s} \,
I_{(p+1,q,s)} \,ds \vspace*{0.15cm}\\
\dis +  \, c^\pm_{\ell, ijk_1...k_m0} b^\pm_{pqL_1\ldots L_m 0} \int
{\cos^i s}\, {\sin^j s} \, I_{(p,q+1,s)} \,ds \bigg),
\end{array}
$$
$$
\begin{array}{l}
\dis \int \frac{\p F_{1\ell}^\pm}{\p z_{\rho}}(s,\varphi(s,\z_\nu))
y^\pm_{1\rho}(s,\z_\nu) ds= \dis \sum_{P = 0}^{n} \,\sum_{Q = 0}^{n} \,
k_\rho \,  r^{i+j+p+q} z_1^{k_1+L_1} ...\,  z_\rho^{k_\rho +L_\rho -1}
...\, z_m^{k_m+L_m} \vspace*{0.15cm}\\
\dis c^\pm_{\ell, ijk_1...k_m0}  \,c^\pm_{\rho,pqL_1 \ldots L_m 0}
\int {\cos^i  s}\,{\sin^j s} \,I_{(p,q,s)} ds,
\end{array}
$$
$$
\begin{array}{l}
\dis \int \frac{\p F_{1\ell}^\pm}{\p z_\omega}(s,\varphi(s,\z_\nu))
y^\pm_{1\omega}(s,\z_\nu) ds= \sum_{P =0}^{n-1}  \,\sum_{Q=0}^{n}
r^{i+j+p+q} z_1^{k_1+L_1} \, z_2^{k_2+L_2} z_3^{k_3+L_3} ... \,
z_m^{k_m+L_m}  \vspace*{0.15cm}\\ \dis c^\pm_{\ell,ijk_1...k_m
1_\omega} c^\pm_{\omega,pqL_1 \ldots L_m 0} \,\int {\cos^i s}
\,{\sin^j s}  \bigg( \int_0^s e^{\mu_\l(s-\tau)} \cos^p \tau \sin^q \tau
d\tau \bigg) ds ,
\end{array}
$$
for $1 \leq \ell \leq m$, $1 \leq \rho \leq m$ and $m+1  \leq \omega
\leq d$.

Moreover, from \eqref{FF} we get
$$
\begin{array}{l}
\dis \int_0^\phi F_{20}^+(s,\z_\nu)\,ds =  \sum_{P=0}^{n}\al^+_{ijk_1
\ldots k_m0} \,r^{i+j} z_1^{k_1}\ldots z_m^{k_m}  \,I_{(i+1,j,\phi)}
\vspace*{0.15cm}\\
+  \dis\sum_{P=0}^{n} \beta^+_{ijk_1\ldots k_m0} \,r^{i+j} z_1^{k_1}
\ldots z_m^{k_m}  \,I_{(i,j+1,\phi)} \\
-\dis\frac{1}{r}\sum_{ P=0}^{n}\,\sum_{Q=0}^{n} a^+_{ijk_1\ldots k_m0}
b^+_{pqL_1\ldots L_m0} \,r^{i+j+p+q}\, z_1^{k_1+L_1} \ldots z_m^{k_m+L_m}
\,I_{(i+p+2,j+q,\phi)}\vspace*{0.15cm}\\
+\dis\frac{1}{r}\sum_{P=0}^{n} \,\sum_{Q=0}^{n} a^+_{ijk_1\ldots k_m0}
a^+_{pqL_1\ldots L_m0} \,r^{i+j+p+q} \,z_1^{k_1+L_1} \ldots z_m^{k_m+L_m}
\,I_{(i+p+1,j+q+1,\phi)}\vspace*{0.15cm}\\
-\dis\frac{1}{r}\sum_{P=0}^{n} \,\sum_{Q=0}^{n} b^+_{ijk_1\ldots k_m0}
b^+_{pqL_1\ldots L_m0} \,r^{i+j+p+q} \,z_1^{k_1+L_1} \ldots z_m^{k_m+L_m}
\,I_{(i+p+1,j+q+1,\phi)}\vspace*{0.15cm}\\
+ \dis\frac{1}{r}\sum_{P=0}^{n}\,\sum_{Q=0}^{n} a^+_{ijk_1\ldots
k_m0} b^+_{pqL_1\ldots L_m0} \,r^{i+j+p+q} \,z_1^{k_1+L_1} \ldots
z_m^{k_m+L_m}  \,I_{(i+p,j+q+2,\phi)},
\end{array}
$$
$$
\begin{array}{l}
\dis\int_\phi^{2\pi} F_{20}^-(s,\z_\nu) \,ds=
\dis\sum_{P=0}^{n}\al^-_{ijk_1\ldots k_m0} \,r^{i+j}  z_1^{k_1}
\ldots z_m^{k_m} \,J_{(i+1,j,\phi)} \vspace*{0.15cm}\\
+  \dis\sum_{P=0}^{n} \beta^-_{ijk_1\ldots k_m0} \,r^{i+j} z_1^{k_1}
\ldots z_m^{k_m}\,J_{(i,j+1,\phi)} \vspace*{0.15cm}\\
-\dis \frac{1}{r}\sum_{ P=0}^{n}\,\sum_{Q=0}^{n}a^-_{ijk_1\ldots k_m0}
b^-_{pqL_1\ldots L_m0} \,r^{i+j+p+q} \,z_1^{k_1+L_1} \ldots z_m^{k_m+L_m}
\,J_{(i+p+2,j+q,\phi)}\vspace*{0.15cm}\\
+\dis \frac{1}{r}\sum_{P=0}^{n} \,\sum_{Q=0}^{n} a^-_{ijk_1\ldots k_m0}
a^-_{pqL_1\ldots L_m0} \,r^{i+j+p+q} \,z_1^{k_1+L_1} \ldots z_m^{k_m+L_m}
\,J_{(i+p+1,j+q+1,\phi)}\vspace*{0.15cm}\\
-\dis\frac{1}{r}\sum_{P=0}^{n} \,\sum_{Q=0}^{n} b^-_{ijk_1\ldots k_m0}
b^-_{pqL_1\ldots L_m0} \,r^{i+j+p+q} \,z_1^{k_1+L_1} \ldots z_m^{k_m+L_m}
\,J_{(i+p+1,j+q+1,\phi)}\vspace*{0.15cm}\\
+\dis \frac{1}{r}\sum_{P=0}^{n} \,\sum_{Q=0}^{n} a^-_{ijk_1\ldots
k_m0} b^-_{pqL_1\ldots L_m0} \,r^{i+j+p+q} \,z_1^{k_1+L_1} \ldots
z_m^{k_m+L_m}  \,J_{(i+p,j+q+2,\phi)},
\end{array}
$$
\begin{align*}
&\int_0^\phi F^+_{2\ell}(s,\z_\nu) \,ds = \sum_{P=0}^{n}\gamma^+_{\ell,
ijk_1\ldots k_m} r^{i+j} z_1^{k_1}\ldots z_m^{k_m}  \,I_{(i,j,\phi)}
\vspace*{0.2cm}\\
&-\frac{1}{r}\sum_{ P=0}^{n} \,\sum_{Q=0}^{n} b^+_{ijk_1\ldots k_m0}
c^+_{\ell,pqL_1\ldots L_m0} r^{i+j+p+q} \,z_1^{k_1+L_1} \ldots z_m^{k_m+L_m}
\,I_{(i+p+1,j+q,\phi)} \vspace*{0.2cm}\\
&+\frac{1}{r}\sum_{P=0}^{n} \,\sum_{Q=0}^{n} a^+_{ijk_1\ldots k_m0}
c^+_{\ell,pqL_1\ldots L_m0} r^{i+j+p+q} \,z_1^{k_1+L_1} \ldots
z_m^{k_m+L_m}  \,I_{(i+p,j+q+1,\phi)},
\end{align*}
\begin{align*}
&\int_\phi^{2\pi} F^-_{2\ell}(s,\z_\nu) \,ds = \sum_{P=0}^{n}\gamma^-_{\ell,
ijk_1\ldots k_m0} r^{i+j} z_1^{k_1}\ldots z_m^{k_m}  \,J_{(i,j,\phi)}
\vspace*{0.2cm}\\
&-\frac{1}{r}\sum_{ P=0}^{n} \,\sum_{ Q=0}^{n} b^-_{ijk_1\ldots k_m0}
c^-_{\ell,pqL_1\ldots L_m0} r^{i+j+p+q}\, z_1^{k_1+L_1} \ldots z_m^{k_m+L_m}
\,J_{(i+p+1,j+q,\phi)} \vspace*{0.2cm}\\
&+\frac{1}{r}\sum_{P=0}^{n} \,\sum_{ Q=0} a^-_{ijk_1\ldots k_m0}
c^-_{\ell,pqL_1\ldots L_m0} r^{i+j+p+q}\, z_1^{k_1+L_1} \ldots
z_m^{k_m+L_m} \,J_{(i+p,j+q+1,\phi)},
\end{align*}
for $1 \leq \ell \leq  m$.

On the other hand when the perturbation is continuous that is, $\phi
=2\pi$, we have
\begin{align*}
& \widetilde{G}_{10}(\nu) = \sum_{\l=m+1}^{d} \Bigg[
\sum_{P=0}^{n-1} r^{i+j}z_1^{k_1} \ldots z_m^{k_m} \,
\vspace*{0.2cm}\\
&\bigg( a^+_{ijk_1\ldots k_m1_{\l}} \int_0^{2\pi}  e^{\mu_\omega s}
\cos^{i+1} s\,\sin^j s \,ds  +  \,b^+_{ijk_1\ldots k_m 1_{\l  }}
\int_0^{2\pi} e^{\mu_\omega s} \cos^i s \sin^{j+1} s \,ds
\big )  \Bigg]  \vspace*{0.2cm}\\
& \Bigg[ \frac{-1}{1-e^{- \mu_\omega 2\pi}} \sum_{Q=0}^{n} r^{i+j}
z_1^{L_1} \ldots z_m^{L_m}~ c^+_{\l,ijL_1\ldots L_m0} \int_0^{2\pi}
e^{-\mu_\omega s} \cos^i s\,\sin^j s\,ds \Bigg],
\end{align*}
and
\begin{align*}
&\int_0^{2\pi} F^+_{20}(s,\z_\nu)\,ds= \sum_{P=0, \,i\,odd, \,j\,even}^{n}
\al^+_{ijk_1\ldots k_m0} r^{i+j} z_1^{k_1} ... \, z_m^{k_m}  \,I_{(i+1,j,2\pi)}\\
&+  \sum_{P=0,\,i\,even, \,j \,odd}^{n} \beta^+_{ijk_1\ldots k_m0}
r^{i+j} z_1^{k_1} ... \, z_m^{k_m}  \,I_{(i,j+1,2\pi)}
\\
&-\frac{1}{r}\sum_{i,j\,even,P=0}^{n}\,\sum_{p,q\,even, Q=0}^{n}
a^+_{ijk_1\ldots k_m0} b^+_{pqL_1\ldots L_m0} r^{i+j+p+q}
z_1^{k_1+L_1} ... \, z_m^{k_m+L_m} I_{(i+p+2,j+q,2\pi)}\\
&-\frac{1}{r}\sum_{ i,j\,odd,P=0}^{n}\,\sum_{p,q\,odd,Q=0}^{n} a^+_{ijk_1
\ldots k_m0} b^+_{pqL_1\ldots L_m0} \,r^{i+j+p+q} z_1^{k_1+L_1} ... \,
z_m^{k_m+L_m} I_{(i+p+2,j+q,2\pi)}\\
&-\frac{1}{r}\sum_{ i\, even,j\,odd,P=0}^{n}\,\sum_{p\, even,q\,odd,Q=0}^{n} a^+_{ijk_1
	\ldots k_m0} b^+_{pqL_1\ldots L_m0} \,r^{i+j+p+q} z_1^{k_1+L_1} ... \,
z_m^{k_m+L_m} I_{(i+p+2,j+q,2\pi)}\\
&-\frac{1}{r}\sum_{ i\, odd,j\,even,P=0}^{n}\,\sum_{p\,odd,q\,even,Q=0}^{n} a^+_{ijk_1
	\ldots k_m0} b^+_{pqL_1\ldots L_m0} \,r^{i+j+p+q} z_1^{k_1+L_1} ... \,
z_m^{k_m+L_m} I_{(i+p+2,j+q,2\pi)}\\
&+\frac{1}{r}\sum_{i,j\,odd, P=0}^{n} \,\sum_{p,q\,even, Q=0}^{n} a^+_{ijk_1
\ldots k_m0} a^+_{pqL_1\ldots L_m0} \,r^{i+j+p+q} z_1^{k_1+L_1} ... \,
z_m^{k_m+L_m} I_{(i+p+1,j+q+1,2\pi)}\\
&+\frac{1}{r}\sum_{i,j\,even, P=0}^{n} \,\sum_{p,q\,odd, Q=0}^{n} a^+_{ijk_1
\ldots k_m0} a^+_{pqL_1\ldots L_m0} \,r^{i+j+p+q} z_1^{k_1+L_1} ... \, z_m^{k_m+L_m}
I_{(i+p+1,j+q+1,2\pi)}\\
&+\frac{1}{r}\sum_{i\, even,j \,odd, P=0}^{n} \,\sum_{p\, odd,q\,even, Q=0}^{n} a^+_{ij
k_1\ldots k_m0} a^+_{pqL_1\ldots L_m0} \,r^{i+j+p+q} z_1^{k_1+L_1} ... \,
z_m^{k_m+L_m}  I_{(i+p+1,j+q+1,2\pi)}\\
&+\frac{1}{r}\sum_{i\, odd,j\,even, P=0}^{n} \,\sum_{p\,even,q\,odd, Q=0}^{n} a^+_{ijk_1
\ldots k_m0} a^+_{pqL_1\ldots L_m0} \,r^{i+j+p+q} z_1^{k_1+L_1} ... \, z_m^{k_m+L_m}
I_{(i+p+1,j+q+1,2\pi)}
\end{align*}
\begin{align*}
&-\frac{1}{r}\sum_{i,j\,odd, P=0}^{n} \,\sum_{p,q\,even, Q=0}^{n} b^+_{ijk_1
	\ldots k_m0} b^+_{pqL_1\ldots L_m0} \,r^{i+j+p+q} z_1^{k_1+L_1} ... \,
z_m^{k_m+L_m} I_{(i+p+1,j+q+1,2\pi)}\\
&-\frac{1}{r}\sum_{i,j\,even, P=0}^{n} \,\sum_{p,q\,odd, Q=0}^{n} b^+_{ijk_1
	\ldots k_m0} b^+_{pqL_1\ldots L_m0} \,r^{i+j+p+q} z_1^{k_1+L_1} ... \, z_m^{k_m+L_m}
I_{(i+p+1,j+q+1,2\pi)}\\
&-\frac{1}{r}\sum_{i\, even,j \,odd, P=0}^{n} \,\sum_{p\, odd,q\,even, Q=0}^{n} b^+_{ij
	k_1\ldots k_m0} b^+_{pqL_1\ldots L_m0} \,r^{i+j+p+q} z_1^{k_1+L_1} ... \,
z_m^{k_m+L_m}  I_{(i+p+1,j+q+1,2\pi)}\\
&-\frac{1}{r}\sum_{i\, odd,j\,even, P=0}^{n} \,\sum_{p\,even,q\,odd, Q=0}^{n} b^+_{ijk_1
	\ldots k_m0} b^+_{pqL_1\ldots L_m0} \,r^{i+j+p+q} z_1^{k_1+L_1} ... \, z_m^{k_m+L_m}
I_{(i+p+1,j+q+1,2\pi)}\\
&+\frac{1}{r}\sum_{i,j\,even, P=0}^{n}\,\sum_{p,q\,even, Q=0}^{n} a^+_{ijk_1\ldots
k_m0} b^+_{pqL_1\ldots L_m0} \,r^{i+j+p+q} z_1^{k_1+L_1} ... \, z_m^{k_m+L_m}
I_{(i+p,j+q+2,2\pi)}\\
&+\frac{1}{r}\sum_{i,j \,odd, P=0}^{n}\,\sum_{p,q\,odd, Q=0}^{n}
a^+_{ijk_1\ldots k_m0} b^+_{pqL_1\ldots L_m0} \,r^{i+j+p+q}
z_1^{k_1+L_1} ... \, z_m^{k_m+L_m}  I_{(i+p,j+q+2,2\pi)}\\
&+\frac{1}{r}\sum_{i\,even,j\,odd, P=0}^{n}\,\sum_{p\,even,q\,odd, Q=0}^{n} a^+_{ijk_1\ldots
	k_m0} b^+_{pqL_1\ldots L_m0} \,r^{i+j+p+q} z_1^{k_1+L_1} ... \, z_m^{k_m+L_m}
I_{(i+p,j+q+2,2\pi)}\\
&+\frac{1}{r}\sum_{i\,odd,j \,even, P=0}^{n}\,\sum_{p\,odd,q\,even, Q=0}^{n}
a^+_{ijk_1\ldots k_m0} b^+_{pqL_1\ldots L_m0} \,r^{i+j+p+q}
z_1^{k_1+L_1} ... \, z_m^{k_m+L_m}  I_{(i+p,j+q+2,2\pi)}.
\end{align*}

\section*{Acknowledgements}

We thank to the referees for their helpful comments and suggestions.

JL  is partially supported by the MINECO/FEDER grants
MTM2016-77278-P and MTM2013-40998-P, and an AGAUR grant
2009SGR-0410. 
DDN is partially supported by FAPESP grant 2018/16430-8 and by CNPq grant 306649/2018-7.
IOZ is partially supported by a
FAPESP grant 2013/21078-8. 
DDN and IOZ are also partially supported by CNPq grant 438975/2018-9. 

\bibliographystyle{abbrv}
\bibliography{references.bib}

\end{document}